\documentclass[final,onefignum,onetabnum]{siamart220329}
\usepackage{braket,amsfonts,amsmath}
\usepackage{booktabs} 
\usepackage[caption=false]{subfig}
\usepackage{tikz-cd}
\usepackage{algorithm}
\usepackage{algpseudocode}
\usepackage{xspace}
\usepackage[mathscr]{euscript}
\usepackage{bold-extra}
\usepackage[most]{tcolorbox}
\usepackage{multirow}
\usepackage{graphicx}
\usepackage{graphics}
\usepackage{float}
\graphicspath{{figures/}{./}}
\usepackage{cleveref}
\usepackage{hyperref}

\newsiamthm{claim}{Claim}
\newsiamremark{remark}{Remark}
\newsiamremark{hypothesis}{Hypothesis}
\crefname{hypothesis}{Hypothesis}{Hypotheses}

\newcommand{\bm}[1]{\mbox{\boldmath $#1$}}

\title{Tensor-based Homogeneous Polynomial Dynamical System Analysis from Data\thanks{Submitted to the editors \today.
\funding{Funding information goes here.}}}

\author{Xin Mao\thanks{School of Data Science and Society, University of North Carolina at Chapel Hill, Chapel Hill, NC 27599, USA (\email{xinm@unc.edu}).}
\and Anqi Dong\thanks{Division of Decision and Control Systems and Department of Mathematics, KTH Royal Institute of Technology, SE-100 44 Stockholm, Sweden (\email{anqid@kth.se}).}
\and Ziqin He\thanks{Department of Mathematics, University of North Carolina at Chapel Hill, Chapel Hill, NC 27599, USA (\email{zhe21@unc.edu}).}
\and Yidan Mei\thanks{Department of Mathematics, University of North Carolina at Chapel Hill, Chapel Hill, NC 27599, USA (\email{ymei@unc.edu}).}
\and Shenghan Mei\thanks{Department of Mathematics, University of North Carolina at Chapel Hill, Chapel Hill, NC 27599, USA (\email{shmei@unc.edu}).}
\and Can Chen\thanks{School of Data Science and Society, Department of Mathematics, and Department of Biostatistics, University of North Carolina at Chapel Hill, Chapel Hill, NC 27599, USA(\email{canc@unc.edu}).}}
\ifpdf
\hypersetup{pdftitle={SIAM Journal on Control and Optimization} }
\fi

\begin{document}
\maketitle

\begin{abstract}
Numerous complex real-world systems, such as those in biological, ecological, and social networks, exhibit higher-order interactions that are often modeled using polynomial dynamical systems or homogeneous polynomial dynamical systems (HPDSs). However, identifying system parameters and analyzing key system-theoretic properties remain challenging due to their inherent nonlinearity and complexity, particularly for large-scale systems. To address these challenges, we develop an innovative computational framework in this article that leverages advanced tensor decomposition techniques, namely tensor train and hierarchical Tucker decompositions, to facilitate efficient identification and analysis of HPDSs that can be equivalently represented by tensors. Specifically, we introduce memory-efficient system identification techniques for directly estimating system parameters represented through tensor decompositions from time-series data. Additionally, we develop necessary and sufficient conditions for determining controllability and observability using the tensor decomposition-based representations of HPDSs, accompanied by detailed complexity analyses that demonstrate significant reductions in computational demands. The effectiveness and efficiency of our framework are validated through numerical examples.
\end{abstract}
	
\begin{keywords}
homogeneous polynomial dynamical systems, system identification, controllability, observability, tensors, tensor train decomposition, hierarchical Tucker decomposition
\end{keywords}
	
\begin{MSCcodes}
15A69, 53A45, 93B05, 93B07, 93B40, 93C15, 93D05
\end{MSCcodes}

\section{Introduction}
Polynomial dynamical systems are pivotal in modeling the dynamics of complex, nonlinear, and higher-order processes, with applications across diverse fields, including biological networks, ecological networks, chemical reactions, epidemiology, robotics, and power systems \cite{bairey2016high, bastolla2009architecture, chellaboina2009modeling, chernyshenko2014polynomial, guckenheimer2013nonlinear, lorenz1963deterministic, malizia2024reconstructing}.  In ecological networks, for instance, species often engage in higher-order interactions, where the relationship between two species is influenced by the presence of one or more additional species. Such dynamics can be effectively captured using a higher-order generalized Lotka-Volterra model, which is a type of polynomial dynamical systems \cite{bairey2016high,bick2023higher}. While understanding polynomial dynamical systems is challenging due to their inherent nonlinearities, they can be simplified into homogeneous polynomial dynamical systems (HPDSs) under certain conditions or transformations (e.g., homogenization). This highlights the need for effective and efficient computational tools to analyze HPDSs. Several system analysis techniques, including stability, controllability, and observability, have been developed for HPDSs \cite{ahmadi2019algebraic, baillieul1981controllability, chen2022explicit, chen2021controllability, jurdjevic1985polynomial, pickard2023observability}. However, these methods are often too complicated to represent and evaluate due to the involvement of convoluted mathematical operations, such as Lie brackets and derivatives.

Recent advances in tensor-based methods have introduced powerful tools for modeling nonlinear dynamics \cite{ batselier2022low,batselier2017tensor,brazell2013solving,cui2024discrete,cui2025metzler,dolgov2021tensor,dong2024controllability,gorodetsky2018high,liu2016near,pickard2023observability}. Tensors, which generalize vectors and matrices to higher dimensions, are well-suited for capturing complex, multidimensional relationships in dynamical systems \cite{kolda2006multilinear}. In particular, every HPDS can be equivalently expressed as a tensor, serving as an elegant analog to matrices in linear dynamical systems \cite{chen2022explicit,chen2024stability}. This tensor-based representation facilitates the exploration of key system properties, such as stability, controllability, and observability, in a compact and mathematically tractable form. For example, the authors of \cite{chen2021controllability} employed tensor algebra to generalize Kalman's rank condition for determining the controllability of tensor-based HPDSs with linear input. Similarly, a tensor-based rank condition for HPDS observability was discussed in \cite{pickard2023observability}. Despite these advances, significant challenges persist.  One major issue is the memory and computational complexity, especially for HPDSs represented by high-dimensional and high-order tensors. Moreover, many existing methods \cite{chen2022explicit,chen2021controllability, pickard2023observability} are limited to specific  HPDSs represented by symmetric tensors, which are inadequate for fully capturing the complex structures and behaviors of general HPDSs in real-world applications. 

To address the gap, this article aims to utilize tensor decomposition to represent tensor-based HPDSs in a more compact form, effectively reducing system dimensionality, lowering memory requirements, and enabling more efficient computations compared to the full tensor representation. There are various tensor decomposition techniques in the literature, such as CANDECOMP/PARAFAC decomposition \cite{bader2008efficient,phan2013candecomp}, higher-order singular value decomposition \cite{de2000multilinear,sidiropoulos2017tensor}, tensor train decomposition (TTD) \cite{oseledets2011tensor,oseledets2009breaking}, and hierarchical tucker decomposition (HTD) \cite{grasedyck2010hierarchical,hackbusch2009new}.  
This work particularly focuses on TTD and HTD for several reasons. Compared to the first two, TTD offers a good balance between numerical stability and approximation accuracy, making it well-suited for high-dimensional tensor computations. Meanwhile, HTD excels at approximating high-dimensional tensors by employing a multi-level approach, which is highly efficient for systems involving intricate structures. Both decompositions have been successfully applied in various dynamic settings. For instance,  TTD was used for model reduction in multi-linear dynamical systems \cite{chen2021multilinear}, while HTD was applied to approximate the coefficient tensor in time-series brain imaging data, demonstrating its ability to handle large-scale, complex datasets effectively \cite{hou2015hierarchical}.

In this article, we develop an innovative computational framework that leverages TTD and HTD to enable efficient identification and analysis of HPDSs. We first introduce memory-efficient system identification techniques by generalizing the fundamental lemma \cite{ van2020data, willems2005note} for tensor-based HPDSs,  enabling the direct estimation of system parameters represented through TTD and HTD from time-series data.  After obtaining the TTD- and HTD-based representations, we derive necessary and sufficient conditions for determining the controllability and observability of input/output tensor-based HPDSs by leveraging the TTD and HTD factor matrices/tensors, which facilitate efficient storage and computation based on the low-rank structures inherent in the original tensors. We further analyze the memory and computational complexity of TTD- and HTD-based methods, comparing them to the full tensor representation-based approaches, and demonstrate their effectiveness and efficiency through numerical simulations. Our framework has significant potential to advance the understanding of complex real-world systems modeled by polynomial dynamical systems or HPDSs. For example, it can  be applied to efficiently capture higher-order species interactions in ecological networks from time-series abundance data (which are often difficult to observe through experiments) and inform more effective designs of optimal control strategies to support ecological conservation efforts.

The article is organized into six sections. In \cref{sec: preliminaries}, we present an overview of tensor preliminaries, including matrix/tensor products, tensor unfolding,  tensor decomposition, and the tensor-based representation of HPDSs. In \cref{sec: sysid}, we develop system identification techniques for tensor-based HPDSs and their associated input/output systems, with a particular focus on the TTD- and HTD-based representations.  We subsequently analyze the controllability and observability of tensor-based HPDSs, deriving efficient conditions using TTD and HTD factor matrices/tensors in  \cref{sec: HPDS}. We validate our framework through numerical simulations in \cref{sec: simulation} and  conclude with future research directions in  \cref{sec:conclusion}.

\section{Preliminaries}\label{sec: preliminaries}
\subsection{Kronecker product}
The Kronecker product plays a significant role in tensor algebra \cite{liu2008hadamard}. Given two matrices of arbitrary sizes $\textbf A\in\mathbb{R}^{n\times m}, \textbf B\in\mathbb{R}^{s\times r}$, the Kronecker product of the two matrices is defined as 
\[
\textbf A\otimes \textbf B=\begin{bmatrix}
\textbf A_{11}\textbf B &\cdots &\textbf A_{1m}\textbf B\\
\vdots& \ddots &\vdots\\
\textbf A_{n1}\textbf B &\cdots &\textbf A_{nm}\textbf B
\end{bmatrix}\in\mathbb{R}^{ns\times mr},
\]
where $\textbf A_{ij}$ denotes the $(i,j)$th element of $\textbf A$.
The Kronecker product has several useful properties, including the mixed-product property which states that for  matrices $\textbf A$, $\textbf B$, $\textbf C$, and  $\textbf D$ with compatible dimensions, the following relationship holds:
\[
(\textbf A\otimes\textbf B)(\textbf C\otimes\textbf D) =(\textbf A\textbf C) \otimes(\textbf B\textbf D).
\]
Given two matrices $\textbf A\in\mathbb{R}^{n\times s}$ and $\textbf B\in\mathbb{R}^{m\times s}$, the Khatri-Rao product of the two matrices is defined as 
\[
\textbf A\odot \textbf B=\begin{bmatrix}
\textbf a_1\otimes \textbf b_1& \textbf a_2\otimes \textbf b_2&\cdots &\textbf a_s\otimes \textbf b_s
\end{bmatrix}\in\mathbb{R}^{nm\times s},
\]
where $\textbf a_j$ and $\textbf b_j$ denote the $j$th column of $\textbf A$ and $\textbf B$, respectively. 

\subsection{Tensors}
A tensor is a multidimensional array that extends the concepts of scalars, vectors, and matrices to higher dimensions \cite{chen2024tensor,kolda2009tensor,kolda2006multilinear}. Mathematically, the order of a tensor refers to the number of its dimensions, and each dimension is called a mode. A $k$th-order tensor is typically denoted by $\mathscr A\in\mathbb{R}^{n_1\times n_2\times\cdots\times n_k}$. 
A tensor is called cubical if all its modes have the same size, i.e., $n_1=n_2=\cdots=n_k$. A cubical tensor is called symmetric if its entries $\mathscr A_{j_1j_2\cdots j_k}$ are invariant under any permutation of the indices, and is called almost symmetric if its entries are invariant under any permutation of the first $k-1$ indices. As an illustrative example,  a fourth-order cubical tensor $\mathscr A\in\mathbb{R}^{n\times n\times n\times n}$ is almost symmetric if 
\begin{align*}
\mathscr A_{j_1j_2j_3j_4}= \mathscr A_{j_1j_3j_2j_4}=\mathscr A_{j_2j_1j_3j_4}
=\mathscr A_{j_2j_3j_1j_4}=\mathscr A_{j_3j_1j_2j_4}=\mathscr A_{j_3j_2j_1j_4}
\end{align*}
for all $j_1, j_2, j_3=1,2,\dots,n$. 

\subsection{Tensor products}	
Given a $k$th-order tensor $\mathscr A\in\mathbb{R}^{n_1\times n_2\times\cdots\times n_k}$ and a vector $\mathbf v\in\mathbb{R}^{n_p}$, the tensor vector multiplication $\mathscr A \times_p \mathbf v\in\mathbb{R}^{n_1\times n_2\times \cdots\times n_{p-1}\times n_{p+1}\times \cdots\times n_k}$ along mode $p$  is defined as 
\[
(\mathscr A \times_p \textbf v)_{j_1j_2\cdots j_{p-1}j_{p+1}\cdots j_k}=\sum_{j_p=1}^{n_p}\mathscr A_{j_1j_2\cdots j_k}\textbf v_{j_p}.
\]
It can be extended across all modes of $\mathscr{A}$ as
\begin{align*}
\mathscr A\times_1 \textbf v_1 \times_2 \textbf v_2\times_3\dots\times_k\textbf v_k \in\mathbb{R},
\end{align*}
where $\textbf{v}_p\in\mathbb{R}^{n_p}$ for $p=1,2,\dots,k$. If $\mathscr{A}$ is symmetric and $\textbf v_p=\textbf v$ for all $p$, the product is referred to as the homogeneous polynomial associated with $\mathscr A$. We write it as $\mathscr A\textbf v^{k}$ for simplicity, i.e., $\mathscr A\textbf v^{k} = \mathscr{A}\times_1\textbf{v}\times_2\textbf{v}\times_3\cdots\times_k\textbf{v}$. Therefore, if $\mathscr{A}$ is almost symmetric, the product $\mathscr A\textbf v^{k-1}$ becomes a homogeneous polynomial system. Furthermore, tensor matrix multiplications can be defined similarly. 

\subsection{Tensor unfolding}
Tensor unfolding, also known as tensor flattening, is a fundamental operation in tensor algebra that reorganizes the entries of a tensor into a matrix or vector representation \cite{kolda2009tensor,ragnarsson2012block}. Given a $k$th-order tensor $\mathscr A\in\mathbb{R}^{n_1\times n_2\times\dots\times n_k}$, we formalize this process by defining an index mapping function 
\vspace{-5pt}
\begin{equation*}
    \psi(\{j_1,j_2,\dots,j_k\},\{n_1,n_2,\dots,n_k\})=j_1+\sum_{i=2}^k(j_i-1)\underset{l=1}{\overset{i-1}{\prod}}n_l.
\vspace{-5pt}
\end{equation*}
This function returns the index for tensor vectorization, i.e., $\text{vec}(\mathscr{A})\in\mathbb{R}^{\prod_{p=1}^k j_p}$ is the vectorization of $\mathscr{A}$ such that $\text{vec}(\mathscr{A})_{\psi} = \mathscr{A}_{j_1j_2\cdots j_k}$.

For tensor matricization, the modes are divided into two disjoint sets with $\mathcal{R}=\{r_1,r_2,\dots,r_d\}$ and $\mathcal{C}=\{c_1,c_2,\dots,c_{k-d}\}$, where the elements in both sets are arranged in increasing order. The tensor matricization therefore is obtained by merging the modes in $\mathcal{R}$ into row indices and the modes in $\mathcal{C}$ into column indices. Specifically, let $r=\psi(\{j_{r_1},\dots,j_{r_d}\},\{n_{r_1},\dots,n_{r_d}\})$ and $c=\psi(\{j_{c_1},\dots,j_{c_{k-d}}\},\{n_{c_1},\dots,n_{c_{k-d}}\})$, and the resulting $\mathcal{R}$-unfolded matrix of $\mathscr{A}$ is defined as
\begin{equation*}
    \textbf{A}_{(\mathcal{R})}\in\mathbb{R}^{(\prod_{p=1}^{d}n_{r_p})\times(\prod_{p=1}^{k-d}n_{c_{p}})} \quad \text{such that}\quad (\textbf{A}_{(\mathcal{R})})_{rc} = \mathscr A_{j_1j_2\cdots j_k}.
\end{equation*}
The $p$-mode matricization of $\mathscr A$ is a special case of tensor unfolding with $\mathcal{R}=\{p\}$ and  $\mathcal{C}=\{1,2,\dots,p-1,p+1,\dots,k\}$. This operation flattens the tensor along mode $p$ into a matrix representation, denoted by $\textbf{A}_{(p)}\in\mathbb{R}^{n_p\times(n_1n_2\cdots n_{p-1}n_{p+1}\cdots n_k)}$. If $\mathscr{A}$ is symmetric, the $p$-mode matricizations are equivalent for all $p$.

\subsection{Tensor decomposition}
Tensor decomposition is a fundamental technique for transforming high-dimensional tensors into lower-dimensional components \cite{kolda2009tensor}. Various tensor decompositions have been proposed, including higher-order singular value decomposition \cite{de2000multilinear}, CANDECOMP/PARAFAC decomposition \cite{bader2008efficient}, tensor train decomposition (TTD) \cite{oseledets2011tensor}, and hierarchical tucker decomposition (HTD) \cite{grasedyck2010hierarchical}. This work focuses on TTD and HTD due to their distinct advantages over the first two. TTD strikes a balance between numerical stability and approximation accuracy, making it effective for high-dimensional tensor computations, while HTD employs a multi-level approach well-suited for complex hierarchical structures. 

\begin{figure}[htbp]
\centering
\begin{minipage}[t]{\textwidth} % 
\centering
\includegraphics[width=0.7\linewidth]{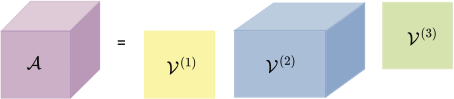} 
\caption{An example of the TTD of a third-order tensor.}
\label{fig: TTD}
\end{minipage}
\\ 
\vspace{0.5cm}
\begin{minipage}[t]{\textwidth} 
\centering
\includegraphics[width=0.55\linewidth]{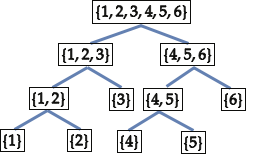} % 
\caption{An example of the HTD of a sixth-order tensor.}
\label{fig: HTD}
\end{minipage}
\vspace{-5pt}
\end{figure}

TTD represents a high-dimensional tensor as a sequence or ``train'' of interconnected third-order tensors, which results in a compact and structured representation. For a $k$th-order tensor $\mathscr A\in\mathbb{R}^{n_1\times n_2\times\dots\times n_k}$, the TTD of $\mathscr{A}$ is expressed as 
\begin{equation}\label{eq:ttd}
\mathscr A=\sum_{j_0=1}^{r_0}\sum_{j_1=1}^{r_1}\cdots\sum_{j_k=1}^{r_k}\mathscr V_{j_0:j_1}^{(1)}\circ \mathscr V_{j_1:j_2}^{(2)}\circ\dots\circ \mathscr V_{j_{k-1}:j_k}^{(k)},
\end{equation}
where $\{r_0, r_1,\dots, r_k\}$ is the set of TT-ranks with $r_0=r_k=1$, and $\mathscr V^{(p)}\in\mathbb{R}^{r_{p-1}\times n_p\times r_p}$ are third-order factor tensors, see \cref{fig: TTD}. The colon ``:'' denotes the MATLAB colon operator acting as shorthand to include all subscripts in a particular array
dimension. Due to the numerical stability, TTD reliably determines the optimal TT-ranks of $\mathscr{A}$, which are computed as $r_p=\text{rank}(\textbf A_{(\{1,2,\dots,p\})})$  where $\textbf A_{(\{1,2,\dots,p\})}$ are the $\{1,2,\dots,p\}$-unfolded matrix of $\mathscr{A}$ for $p=1,2,\dots,k-1$.

HTD employs a tree-structured factorization to represent a tensor recursively, splitting its dimensions at each level to construct factor matrices.  HTD relies on the concept of a dimension tree $\mathcal {T}$, which is a binary tree with the following properties: (i) each node is represented by a subset of $\{1,2,\dots,k\}$; (ii) the root node is $\{1,2,\dots,k\}$; (iii) each leaf node is a singleton; (iv) each parent node is the disjoint union of its two children nodes, as shown in \cref{fig: HTD}. The level of a node in the binary tree is defined as the number of edges on the path from the root node to that node. The root node is considered to be at level zero. The level of the binary tree is the maximum level among all nodes in the tree, given by $d=\lceil\log_2k\rceil$.

Given a $k$th-order tensor $\mathscr A\in\mathbb{R}^{n_1\times n_2\times\dots\times n_k}$, the binary tree encodes a hierarchy of factor matrices $\textbf U_\mathcal Q$, each of which serves as a basis for the column space of the $\mathcal{Q}$-unfolded matrix $\textbf A_{(\mathcal Q)}$ for a node $\mathcal{Q}\in\mathcal{T}$.  Importantly, not all factor matrices in the HTD of $\mathscr{A}$ need to be explicitly stored as the matrix $\textbf U_\mathcal Q$ of a parent node can be computed from its left   child node $\textbf U_{\mathcal Q_l}$ and right child node $\textbf U_{\mathcal Q_r}$ as
\begin{align}\label{eq: HTrelation}
\textbf U_\mathcal Q=(\textbf U_{\mathcal Q_l}\otimes \textbf U_{\mathcal Q_r})\textbf G_{\mathcal Q},
\end{align}
where $\textbf{G}_{\mathcal Q}\in\mathbb{R}^{r_{\mathcal Q_l}r_{\mathcal Q_r}\times r_{\mathcal Q}}$ is called the transfer function with hierarchical ranks $r_{\mathcal Q_l}$, $r_{\mathcal Q_r}$, and $r_\mathcal Q$. Therefore, only the factor matrices $\textbf{U}_{p}\in\mathbb{R}^{n_p\times r_p}$ at the leaf nodes and the transfer matrices $\textbf{G}_{\mathcal{Q}}$ at different levels need to be  stored. The construction of the HTD starts from the leaf nodes and proceeds recursively by applying \eqref{eq: HTrelation} until the root node is reached, where the hierarchical rank at the root is fixed to one. Additionally,  the HTD of $\mathscr A$ in its vectorized form  can be expressed as
\[
\text{vec}({\mathscr A})=(\textbf U_{k}\otimes\textbf U_{k-1}\otimes\cdots\otimes\textbf U_{1})(\otimes_{\mathcal Q\in\mathcal G_{d-1}}\textbf G_\mathcal Q)\cdots(\otimes_{\mathcal Q\in\mathcal G_{0}}\textbf G_\mathcal Q),
\]
where $\mathcal G_j$ denote the set of nodes at level $j$ of the binary tree for $j=0,1,\dots,d-1$. 

\subsection{Tensor-based HPDSs}
Every homogeneous polynomial dynamical system (HPDS) of degree $k-1$ can be  represented using tensor vector multiplications as 
\begin{align}\label{eq: ausys}
\dot{\textbf x}(t)=\mathscr A\textbf x(t)^{k-1},
\end{align}
where $\mathscr A\in\mathbb{R}^{n\times n\times\stackrel{k}{\cdots}\times n}$ is a $k$th-order $n$-dimensional almost symmetric dynamic tensor and $\textbf x(t)\in\mathbb{R}^n$ is the state \cite{chen2022explicit,chen2024stability}. The notation $``n\times n\times\stackrel{k}{\cdots}\times n"$ represents a total of $k$ multiplications of $n$. This compact tensor-based representation introduces a new perspective for analyzing HPDSs by leveraging advanced tensor theory. Significantly, the tensor-based HPDS (\ref{eq: ausys}) can be equivalently written as 
\begin{align}\label{eq: ausysm}
\dot{\textbf x}(t)=\textbf A_{(k)}\textbf x(t)^{[k-1]},
\end{align}
where $\textbf A_{(k)}$ is the $k$-mode matricization of $\mathscr A$, and $\textbf{x}(t)^{[k-1]}$ denotes the $(k-1)$th Kronecker power, i.e., $\textbf{x}^{[k-1]}=\textbf x\otimes \textbf x\otimes\stackrel{k-1}{\cdots}\otimes \textbf x$. 

\section{System identification for tensor-based HPDSs}\label{sec: sysid}
Traditional linear system identification methods are well-established and widely used in various applications \cite{ljung1998system,ljung2010perspectives,nelles2020nonlinear,reynders2012system}. However, many real-world systems exhibit significant nonlinearities that linear models fail to capture adequately. To address this gap, we systematically investigate the identification of tensor-based HPDSs, focusing first on autonomous systems and extending the framework to input-output systems with noise.  A key feature of our framework is the direct estimation of system parameters within the TTD and HTD  of the dynamic tensor $\mathscr{A}$, leveraging time-series data to achieve more structured and scalable representations of  HPDSs.

\subsection{Identification of autonomous system}
We begin by addressing system identification in the absence of external inputs, a scenario common in autonomous systems. Our objective is to generalize the fundamental lemma and facilitate direct parameter estimation from observed data. To this end, we collect the sampled state and  the derivatives of the state data during the time interval $[t_0, t_0+(T-1)\tau]$ as 
\begin{align*}
\textbf X_0&=\begin{bmatrix}
\textbf x(t_0) & \textbf x(t_0+\tau) & \cdots & \textbf x(t_0+(T-1)\tau)\end{bmatrix}\in\mathbb{R}^{n\times T},\\
\textbf X_1&=\begin{bmatrix}
\dot{\textbf x}(t_0) & \dot{\textbf x}(t_0+\tau) & \cdots & \dot{\textbf x}(t_0+(T-1)\tau)\end{bmatrix}\in\mathbb{R}^{n\times T},
\end{align*}
where $\tau$ is the sampling time and $T$ is the number of the sampled data. We seek to establish a condition that allows for the unique identification of the dynamic tensor $\mathscr A$.	
Denote 
$
\hat{\textbf X}_0 =\textbf X_0\odot \textbf X_0\odot\stackrel{k-1}{\cdots}\odot \textbf X_0.
$
A condition similar to Willems's fundamental lemma is presented below.

\begin{proposition}
\label{thm: ausysid}
The data $(\textbf X_0, \textbf X_1)$ is sufficient for unique identification of the tensor-based HPDS (\ref{eq: ausys}) if and only if
\begin{align}\label{aurankcondition}
\mathrm{rank}(\hat{\textbf X}_0)=\sum_{j=1}^{\mathrm{min}\{n,k-1\}}\frac{n!}{j!(n-j)!}\frac{(k-2)!}{(j-1)!(k-j-1)!}.
\end{align}
\end{proposition}
\begin{proof}
As the tensor-based HPDS \eqref{eq: ausys} is equivalent to the unfolded system \eqref{eq: ausysm}, the object is to check whether $\textbf A_{(k)}$ can be uniquely identified from observed data. The relation of data $(\textbf X_0,\textbf X_1)$ is $\textbf A_{(k)}\hat {\textbf X}_0=\textbf X_1.$
Since $\mathscr A$ is an almost symmetric tensor, each row of $\textbf A_{(k)}$ has $n^{k-1}$ elements and $\sum_{j=1}^{\mathrm{min}\{n,k-1\}}\frac{n!}{j!(n-j)!}\frac{(k-2)!}{(j-1)!(k-j-1)!}$ independent entries. According to linear matrix theory, the necessary and sufficient condition for the unique identification of $\textbf A_{(k)}$ is that $\hat{\textbf X}_0$ has full row rank, i.e.,
\[
\mathrm{rank}(\hat{\textbf X}_0)
=\sum_{j=1}^{\mathrm{min}\{n,k-1\}}\frac{n!}{j!(n-j)!}\frac{(k-2)!}{(j-1)!(k-j-1)!}.
\]
The result follows immediately.
\end{proof}

\cref{thm: ausysid} not only provides conditions for system identification but also ensures the uniqueness of $\textbf A_{(k)}$ as the solution consistent with the observed data. When $k=2$, the result reduces to the classical Willems’s fundamental lemma.  Moreover, $\textbf A_{(k)}$ can be  computed using compact singular value decomposition (SVD), similar to dynamic model decomposition in linear dynamical systems.

\begin{proposition}\label{thm: auconstruct}
If condition (\ref{aurankcondition}) holds  and the compact SVD of $\hat{\textbf X}_0$ is provided with $\hat{\textbf X}_0=\hat{\textbf U}_0\hat{\bm\Sigma}_0\hat{\textbf V}^\top_0$, then $\textbf A_{(k)}=\textbf X_1\hat{\textbf V}_0\hat{\bm\Sigma}_0^{+}\hat{\textbf U}_0^\top$. 
\end{proposition}
\begin{proof}
Since $\hat{\textbf X}_0 =\textbf X_0\odot \textbf X_0\stackrel{k-1}\cdots\odot \textbf X_0$, the matrix $\hat{\textbf X}_0$ contains duplicate rows. Specifically, for a given set $\mathcal J=\{j_1, j_2,\cdots,j_{k-1}\}$, the rows indexed by 
\begin{equation*}
    i=\sigma(j_1)+\sum\limits_{m=2}^{k-1}(\sigma(j_m)-1) n^{m-1}
\end{equation*}
are identical for all possible permutations $\sigma(\cdot)$ of elements in $\mathcal J$. To show that $\textbf X_1\textbf V_0\bm\Sigma_0^+\textbf U_0^\top$ corresponds to the $k$-mode unfolding of an almost symmetric tensor, it suffices to prove that the $i$th columns of $\textbf V_0\bm\Sigma_0^+\textbf U_0^\top$ are equal for indices given by the formula above. This holds if and only if the corresponding rows of $\textbf U_0$, derived from the compact SVD, are identical. We analyze the structure of $\textbf U_0$. Let $\textbf w$ denote the common value of the repeated rows of $\hat{\textbf X}_0$ and let $h$ be the number of such duplicate rows. The corresponding rows of $\textbf U_0$ are denoted by $\textbf u_1,\textbf u_2,\dots, \textbf u_l$. From the SVD, we have $\textbf u_p\bm\Sigma_0\textbf V_0^\top=\textbf w$ for $p=1,2,\dots,h$. Now assume  for contradiction that $\textbf u_p\neq \textbf u_q$ for some $p\neq q$. Then, taking the difference between the above equations for $p$ and $q$, we obtain $(\textbf u_p-\textbf u_q)\bm\Sigma_0\textbf V_0^\top=0$. Since $\bm\Sigma_0$ is a non-singular diagonal matrix and $\textbf V_0$ is an isometry, their product $\bm\Sigma_0\textbf V_0^\top$ is full rank. This implies that the only solution is $\textbf u_p=\textbf u_q$, contradicting our assumption. Thus, all $h$ rows of $\textbf U_0$ must be equal, and the result follows immediately. 
\end{proof}

After solving $\textbf{A}_{(k)}$,  the dynamic tensor $\mathscr{A}$ by can be readily obtained by applying the reverse $k$-mode matricization. However, this full tensor representation imposes a substantial memory burden, with the number of parameters scaling as  $\mathcal{O}(n^k)$. To mitigate this, we adopt TTD and HTD to restructure high-dimensional, high-order tensor-based HPDSs into compact, low-dimensional formats. Both representations significantly reduce storage requirements and computational complexity, making them well-suited for large-scale HPDS identification.

\subsection{Identification of TTD- and HTD-based HPDSs}
We focus on the data-driven identification of TTD and HTD factor tensors/matrices of the dynamic tensor $\mathscr{A}$, enabling the direct extraction of low-dimensional HPDS representations from observed data.  Given the TTD of a tensor, multiplying the tensor by a vector along a mode is equivalent to applying a linear transformation to the corresponding factor tensor in that mode. Therefore, for the tensor-based HPDS (\ref{eq: ausys}), the TTD-based representation can be compactly expressed as
\begin{align}\label{eq:ttdhpds}
\dot{\textbf x}(t)=\left[\prod_{p=1}^{k-1}\left(\mathscr V^{(p)}\times_2 \textbf x(t)\right)\mathscr V^{(k)}\right]^\top,
\end{align}
where $\mathscr V^{(p)}\in\mathbb{R}^{r_{p-1}\times n\times r_p}$ are the third-order factor tensors from the TTD of $\mathscr{A}$. 

\begin{remark}
If $r=\text{max}\{r_p \text{ }|\text{ } p=1,2,\dots,k-1\}$ is the maximum TT-rank from the TTD of $\mathscr{A}$, the total number of parameters in the TTD-based HPDS \eqref{eq:ttdhpds} can be estimated as $\mathcal O(knr^2).$ 
\end{remark}

\begin{algorithm}[t]
\caption{Data-driven TTD-based HPDS identification}\label{alg: TTD}
\begin{algorithmic}[1]
\Require Sampled state data $\textbf X_0$ and $\textbf X_1$
\Ensure Third-order factor tensors $\mathscr V^{(1)},\mathscr{V}^{(2)},\dots,\mathscr V^{(k)}$
\State Compute $\hat{\textbf X}_0 =\textbf X_0\odot \textbf X_0\odot\stackrel{k-1}{\cdots}\odot \textbf X_0$
and  perform compact SVD $\hat{\textbf X}_0=\hat{\textbf U}_0\hat{\bm\Sigma}_0\hat{\textbf V}_0^\top$ 
\State Initialize $\textbf M=\textbf X_1\hat{\textbf V}_0\hat{\bm\Sigma}_0^\top\hat{\textbf U}_0^\top$ and $r_k=1$.
\For {$p=k$ to $2$}
\State Perform compact SVD $\textbf M=\textbf U_p\bm\Sigma_p\textbf V_p^\top$ and set $r_{p-1}=\mathrm{rank}(\textbf M)$
\State Compute the factor tensor $\mathscr V^{(p)}=\mathrm{reshape}(\textbf U_p^\top,[r_{p-1}, n, r_p])$
\State Update $\textbf M=\bm\Sigma_p\textbf V_p^\top$
\If{$p\neq 2$}
\State Set $\textbf M=\text{reshape}(\textbf M, [r_{p-1}n,n^{p-2}])$
\EndIf
\EndFor  
\State Set $\mathscr V^{(1)}=\textbf M^\top$
\end{algorithmic}
\end{algorithm}
\vspace{-5pt}

This linear scaling in the order $k$ and quadratic scaling in the rank $r$ make the TTD-based representation particularly memory-efficient for high-dimensional HPDSs. Notably, when the maximum TT-rank $r$ is small, the TTD-based HPDS \eqref{eq:ttdhpds} achieves significant memory savings compared to the full tensor representation. A data-driven method for identifying the TTD-based HPDS factor tensors using the sampled state data $\textbf X_0$ and $\textbf X_1$ is presented in \cref{alg: TTD}. Similarly, tensor vector multiplications can be achieved efficiently in the HTD format, which results in the HTD-based representation of the tensor-based HPDS (\ref{eq: ausys}) expressed as
\begin{align}\label{eq:htdhpds}
\dot{\textbf x}(t)=\left(\textbf U_k\otimes\textbf x(t)^\top\textbf U_{k-1}\otimes\cdots\otimes\textbf x(t)^\top\textbf U_{1}\right)(\otimes_{\mathcal Q\in\mathcal G_{d-1}}\textbf G_\mathcal Q)\cdots(\otimes_{\mathcal Q\in\mathcal G_{0}}\textbf G_\mathcal Q).
\end{align}
Due to the almost symmetric nature of $\mathscr A$, it holds that $\textbf U_1=\textbf U_2=\dots=\textbf U_{k-1}$. 

\begin{remark}
If $r=\text{max}\{r_\mathcal Q \text{ }|\text{ } \mathcal Q\in\mathcal T\}$ is the maximum hierarchical  rank from the HTD of $\mathscr{A}$, the total number of parameters in the HTD-based HPDS \eqref{eq:htdhpds} can be estimated as $\mathcal O(knr+kr^3).$
\end{remark}

Likewise, the linear scaling with respect to the order $k$ and cubic scaling with respect to the rank $r$ enable the HTD-based representation to be space-conserving for high-dimensional HPDSs. If the maximum hierarchical rank $r$ is small, the HTD-based HPDS \eqref{eq:htdhpds} also offers substantial memory savings over the full tensor representation. \cref{alg: HTD} presents a data-driven method for identifying the factor and transfer matrices of the HTD-based HPDS \eqref{eq:htdhpds} directly from sampled state data. To summarize, both TTD and HTD provide distinct advantages for representing high-dimensional tensor-based HPDSs. TTD excels in capturing sequential dependencies, factorizing the dynamic tensor into a chain-like structure that enhances memory efficiency and numerical stability. On the other hand, HTD is well-suited for the dynamic tensor with intrinsic hierarchical relationships, offering a tree-like decomposition that captures multi-scale structures while maintaining compactness. Both representations play a crucial role in facilitating HPDS analysis, leveraging the identified factor tensors and matrices to assess properties such as controllability and observability.

\begin{algorithm}[t]
\caption{Data-driven HTD-based HPDS  identification}\label{alg: HTD}
\begin{algorithmic}[1]
\Require Sampled state data $\textbf X_0$, $\textbf X_1$, and a dimension tree $\mathcal T$
\Ensure Factor and transfer matrices $\{\textbf U_{\mathcal Q}, \textbf G_{\mathcal Q}\}$
\State Compute $\hat{\textbf X}_0 =\textbf X_0\odot \textbf X_0\odot\stackrel{k-1}{\cdots}\odot \textbf X_0$ and perform compact SVD $\hat{\textbf X}_0=\hat{\textbf U}_0\hat{\bm\Sigma}_0\hat{\textbf V}_0^\top$
\State Set $\textbf M=\textbf X_1\hat{\textbf V}_0\hat{\bm\Sigma}_0^\top\hat{\textbf U}_0^\top$ and obtain $\textbf U_k$ from compact SVD $\textbf M=\textbf U_k\bm\Sigma_k\textbf V_k^\top$
\State Set $\mathscr M=\text{reshape}(\textbf M, [n,n,\dots,n]$ and  $\textbf M=\text{reshape}(\mathscr M,[n,n^{k-1}])$
\State Obtain $\textbf U_1$ from $\textbf M=\textbf U_1\bm\Sigma_1\textbf V_1^\top$ and set $\textbf U_p=\textbf U_1$ for $p=1,2,\dots,k-1$
\For {$l=\lceil\log_2k\rceil$ to $1$}
\For {each node $\mathcal Q\in\mathcal{T}$ on level $l$}
\State Set $\mathcal S=\text{setdifference}(1:k, \mathcal Q)$.
\State Set $\textbf M_{\mathcal Q}=\text{reshape}(\text{permute}(\mathscr M,[\mathcal Q,\mathcal S]),[n^{size(\mathcal Q)}, n^{size(\mathcal S)}])$
\State Compute compact SVD $\textbf M=\textbf U_{\mathcal Q}\bm\Sigma_{\mathcal Q}\textbf V_{\mathcal Q}^\top$
\State Obtain left child node $\mathcal Q_l$ and right child node $\mathcal Q_  r=\text{setdifference}(\mathcal Q, \mathcal Q_l)$
\State Obtain transfer matrix $\textbf G_\mathcal Q=(\textbf U_{\mathcal Q_l}^\top\otimes \textbf U_{\mathcal Q_r}^\top)\textbf U_\mathcal Q$
\EndFor
\EndFor
\end{algorithmic}
\end{algorithm}
\vspace{-5pt}

\subsection{Identification with input and output data} 
We next address the identification of tensor-based HPDSs using input-output data, where only partial observations of the state are available through output measurements. We consider the following tensor-based HPDS of degree $k-1$ with linear input and output
\begin{align}\label{eq: syso}
\begin{cases}
\dot{\textbf x}(t)=\mathscr A\textbf x(t)^{k-1}+\textbf B\textbf u(t) \\
\textbf y(t)=\textbf C\textbf x(t)
\end{cases},
\end{align}
where $\mathscr A\in\mathbb{R}^{n\times n\times\stackrel{k}{\cdots}\times n}$ is an almost symmetric dynamic tensor, $\textbf B\in\mathbb{R}^{n\times m}$ is the control matrix, $\textbf u(t)\in\mathbb{R}^m$ is the control input, $\textbf C\in\mathbb{R}^{l\times n}$ is the output matrix, and $\textbf y(t)\in\mathbb{R}^l$ is the output. We assume $l \geq n$, and the initial state  $\textbf x(0)$ is known. Our goal centers on estimating $\mathscr A$, $\textbf B$, and $\textbf C$ from the input and output data. To facilitate the analysis, we rewrite \eqref{eq: syso} into the unfolded form 
\begin{align}
\begin{cases}
\dot{\textbf x}(t)=\textbf A_{(k)}\textbf x(t)^{[k-1]}+\textbf B\textbf u(t).\\
\textbf y(t)=\textbf C\textbf x(t).
\end{cases}
\end{align}
Here, we assume that the derivatives $\dot{\textbf x}(t)$ are not directly measurable at any given time $t$. Instead, we approximate them by using sampled output data.

We collect a sequence of input and output data over a sampling time interval $[t_0, t_0+(T-1)\tau]$, forming the following data matrices:
\begin{align*}
\textbf U_0&=\begin{bmatrix}
\textbf u(t_0) & \textbf u(t_0+\tau) & \cdots & \textbf u(t_0+(T-1)\tau)
\end{bmatrix},\\
\textbf Y_0&=\begin{bmatrix}
\textbf y(t_0) & \textbf y(t_0+\tau) & \cdots & \textbf y(t_0+(T-1)\tau)
\end{bmatrix},
\end{align*}
where $\tau$  denotes the sampling time and $T$ is the total number of samples. Using the compact SVD of $\textbf Y_0$ as $\textbf Y_0=\textbf U\bm\Sigma\textbf V^\top$ with $\textbf U\in\mathbb{R}^{l\times n}$, we estimate the output matrix $\textbf C$ as $\textbf C=\textbf U$. The corresponding state matrix is then approximated as
\begin{align*}
\begin{bmatrix}
\textbf x(t_0) & \textbf x(t_0+\tau) & \cdots & \textbf x(t_0+(T-1)\tau)\end{bmatrix}=\bm\Sigma \textbf V^\top.
\end{align*}
Since data is collected in discrete-time, the continuous-time system is discretized using a finite difference approximation
\[
\textbf x[k+1]=\textbf x[k]+\tau A_{(k)}\textbf x[k]^{[k-1]}+\textbf B\textbf u[k].
\]
We then form the state matrices
\begin{align*}
\textbf X_0&=\begin{bmatrix}
\textbf x(t_0) & \textbf x(t_0+\tau) & \cdots & \textbf x(t_0+(T-2)\tau)\end{bmatrix},\\
\textbf	X_1&=\begin{bmatrix}
\textbf x(t_0+\tau) & \textbf x(t_0+2\tau) & \cdots & \textbf x(t_0+(T-1)\tau)\end{bmatrix}.
\end{align*}
Similarly, to account for the polynomial dynamics, we construct the Khatri-Rao product of $\textbf{X}_0$ as
$\hat{\textbf{X}}_0 =\textbf{X}_0\odot \textbf{X}_0\odot\stackrel{k-1}{\cdots}\odot \textbf{X}_0$.

\begin{proposition}\label{thm: iosysid}
The data $(\textbf U_0,\textbf Y_0)$ is sufficient for identification of the tensor-based HPDS with linear input and output (\ref{eq: syso}) if and only if 
\begin{align}\label{iorankcondition}
\mathrm{rank}(\mathbf Y_0)=n \text{ and } \mathrm{rank}\begin{bmatrix}\hat{\textbf X}_0\\
\textbf U_0
\end{bmatrix}=\sum_{j=1}^{\mathrm{min}\{n,k-1\}}\frac{n!}{j!(n-j)!}\frac{(k-2)!}{(j-1)!(k-j-1)!}+m.
\end{align}
\end{proposition}
\begin{proof}
The goal is to determine whether the collected data are sufficient for identifying the matrices $\left(\textbf A_{(k)}, \textbf B,\textbf C\right)$. First compute the compact SVD of $\textbf Y_0$. Since $\textbf U\in\mathbb{R}^{m\times n}$ is equivalent to $\mathrm{rank}(\textbf Y_0)=n$, $\textbf C$ can be identified. Moreover, the relation of state data and input data is given by
\begin{align*}
\textbf X_1&=\textbf X_0+\tau \textbf A_{(k)}\hat{\textbf X}_0+\textbf B\textbf U_0=\textbf X_0+\begin{bmatrix}
\textbf A_{(k)}& \textbf B
\end{bmatrix}\begin{bmatrix}
\tau\hat{\textbf X}_0 \\ \textbf U_0 
\end{bmatrix}
\end{align*}
Each row of $\begin{bmatrix}
\textbf A_{(k)}& \textbf B
\end{bmatrix}$ has $n^{k-1}+m$ elements and $\sum_{j=1}^{\mathrm{min}\{n,k-1\}}\frac{n!}{j!(n-j)!}\frac{(k-2)!}{(j-1)!(k-j-1)!}+m$ independent entries given that $\mathscr A$ is an almost symmetric tensor. The result then follows from \cref{thm: ausysid}.
\end{proof}

When the conditions in \cref{thm: iosysid} are satisfied, the system matrices can be obtained from data as shown in following proposition.
\begin{proposition}
Given condition (\ref{iorankcondition}) with the compact SVD $\textbf Y_0=\textbf U\bm\Sigma\textbf V^\top$ and $\begin{bmatrix}
\hat{\textbf X}_0 & \textbf U_0
\end{bmatrix}^\top=\textbf U_1\bm\Sigma_1 \textbf V^\top_1$, it holds that
$\textbf C=\textbf U$,  $\textbf A_{(k)}=\textbf X_1\textbf T_1$, and $\textbf B=\textbf X_1\textbf T_2$, where $\begin{bmatrix}\textbf T_1 & \textbf T_2\end{bmatrix}=\textbf V_1\bm\Sigma_1^{+}\textbf U_1^\top$.
\end{proposition}
\begin{proof}
The proof is similar to the proof of \cref{thm: auconstruct}
\end{proof}

In practical settings, the collected measurements are frequently corrupted by noise. As a result, the observed data no longer perfectly adheres to the underlying system dynamics. To account for these imperfections, we incorporate noise into  the tensor-based HPDS with linear input and output \eqref{eq: syso}, i.e., 
\begin{align}\label{eq: noise}
\begin{cases}
\dot{\textbf x}(t)=\mathscr A\textbf x(t)^{k-1}+\textbf B\textbf u(t)+ \textbf{w}_1(t) \\
\textbf y(t)=\textbf C\textbf x(t)+\textbf{w}_2(t)
\end{cases},
\end{align}
where  $\textbf{w}_1(t)$ and $\textbf{w}_2(t)$ are independent, identically distributed noises. A filter, such as the Kalman filter \cite{chui2017kalman}, can be used to estimate the state data $\textbf X_0$, $\textbf X_1$ given the noisy measurements $\textbf Y_0$ and the inputs $\textbf U_0$. For Gaussian noise, the system identification problem reduces to a multivariate linear regression problem under the Frobenius norm
\begin{equation*}
    \min \|\textbf{X}_1-\textbf{X}_0-\tau\textbf{A}_{(k)} \hat{\textbf{X}}_0 - \textbf{B} \textbf{U}_0 \|^2_F+\|\textbf Y_0-\textbf C\textbf X_0 \|^2_F.
\end{equation*}
It can be similarly shown that the data $(\textbf U_0,\textbf Y_0)$ is sufficient for identification of the input-output tensor-based HPDS with noise (\ref{eq: noise}) if and only if condition \eqref{iorankcondition} holds. More importantly, all results including input-output identification in both noisy and noise-free settings can be extended to the associated TTD- and HTD-based representations by utilizing \cref{alg: TTD} and \cref{alg: HTD}.

\section{Tensor-based HPDS analysis}\label{sec: HPDS}
Traditional controllability and observability analyses of HPDSs pose significant challenges due to their inherent nonlinearity and the need to evaluate intricate geometric objects such as Lie brackets and Lie derivatives. To overcome these challenges, we present an effective framework for analyzing the controllability and observability of input/output tensor-based HPDSs.  In particular, we derive computationally efficient algebraic conditions by leveraging the structured representations of TTD and HTD factor matrices and tensors, making the analyses more scalable and practical for high-dimensional HPDSs.

\subsection{Controllability}\label{sec: controllability}
We begin by examining the controllability of the following tensor-based HPDS of degree $k-1$ with linear input 
\begin{align}\label{eq: csys}
\dot{\textbf x}(t)=\mathscr A\textbf x(t)^{k-1}+\textbf B\textbf u(t),
\end{align}
where $\mathscr A\in\mathbb{R}^{n\times n
\times\stackrel{k}{\cdots}\times n}$ is an almost symmetric dynamic tensor, and $\textbf B\in\mathbb{R}^{n\times m}$ is the control matrix. Significantly, for HPDSs with linear input, a global controllability condition can be explicitly derived using Lie algebra theory \cite{jurdjevic1985polynomial}. We first introduce the concepts of strong controllability and Lie brackets.

\begin{definition} [\cite{jurdjevic1985polynomial}]
A dynamical system is said to be strongly controllable if, with an appropriate choice of control inputs, it can be driven from any initial state to any target state within any time interval.
\end{definition}
    
For  HPDSs of odd degrees with linear input, strong controllability can be assessed by analyzing the Lie algebra generated by the vector fields associated with the system dynamics and the columns of the control matrix. The Lie algebra can be constructed using recursive Lie brackets. Given smooth vector fields $\textbf f(\textbf{x})$ and $\textbf g(\textbf{x})$, their Lie bracket at point $\textbf x$ is defined as
\[
[\textbf f, \textbf g]_\textbf x=\nabla \textbf g(\textbf x)\textbf f(\textbf x)-\nabla \textbf f(\textbf x)\textbf g(\textbf x),
\]
where $\nabla \textbf g(\textbf x)$ and $\nabla \textbf f(\textbf x)$ are the Jacobian matrices of $\textbf g$ and $\textbf f$ at $\textbf x$, respectively. 
For even-degree systems, however, the problem of global controllability remains open, and only local accessibility can be established \cite{aeyels1984local,chen2021controllability,melody2003nonlinear}. Moreover, 
a tensor-based controllability criterion was established for tensor-based HPDSs of odd degrees with linear input where the dynamic tensors $\mathscr{A}$ are symmetric \cite{chen2021controllability}. However, this result does not encompass more general cases. We therefore extend it to odd-degree tensor-based HPDSs with $\mathscr A$ being almost symmetric.

\begin{proposition}\label{prop: controllability}
The tensor-based HPDS with linear input (\ref{eq: csys}) for even $k$ is strongly controllable if and only if the controllability matrix defined as
\begin{align}
\textbf R=\begin{bmatrix}\textbf M_0 & \textbf M_1 &\cdots& \textbf M_{n-1}\end{bmatrix},
\end{align}
where $\textbf M_0=\textbf B$ and 
each $\textbf M_j$ is formed from 
\begin{equation*}
    \{\mathscr A\textbf v_1\textbf{v}_2\cdots\textbf v_{k-1}\text{ }|\text{ }\textbf v_{p}\in\mathrm{col}\left(\begin{bmatrix}\textbf M_0 & \textbf M_1 &\cdots& \textbf M_{j-1}\end{bmatrix}\right) \text{ for } p=1,2,\dots,k-1\}
\end{equation*}
for $j=1,2,\dots,n-1$, has full rank $n$.
\end{proposition}
\begin{proof}
As demonstrated in \cite{jurdjevic1985polynomial}, an HPDS of odd degree is strongly controllable if and only if the Lie algebra spanned by the set of vector fields $\{\mathscr{A} \mathbf{x}^{k-1}, \mathbf{b}_1, \mathbf{b}_2, \dots, \mathbf{b}_m\}$ has full rank at the origin, where $\mathbf{b}_j$ denotes the $j$th column of $\mathbf{B}$. We show that $\mathbf{R}$ contains all recursive Lie brackets of the vector fields $\{\mathscr{A} \mathbf{x}^{k-1}, \mathbf{b}_1, \mathbf{b}_2, \dots, \mathbf{b}_m\}$ at the origin. Without loss of generality, assume $m = 1$ and denote $\mathbf{b} = \mathbf{b}_1$. Since $\mathscr{A}$ is almost symmetric, the Lie bracket computations follow a structured pattern:
\begin{align*}
[\textbf b, \mathscr A\textbf x^{k-1}]_\textbf{0}&=\left(\left.\frac{d}{d\textbf x}\right |_{\textbf x=\textbf{0}}\mathscr A\textbf x^{k-1}\right)\textbf b=\textbf{0},\\
[\textbf b,	[\textbf b, \mathscr A\textbf x^{k-1}]]_\textbf{0}&=\left(\left.\frac{d}{d\textbf x}\right |_{\textbf x=\textbf{0}}\mathscr A\textbf x^{k-2}\textbf b\right)\textbf b=\textbf{0},\\
\vdots\\
[\textbf b,	[\dots,[\textbf b, \mathscr A\textbf x^{k-1}]]]_{\textbf{0}}&=\left(\left.\frac{d}{d\textbf x}\right |_{\textbf x=0}\mathscr A\textbf x\textbf b^{k-2}\right)\textbf b=\mathscr A\textbf b^{k-1}.
\end{align*}
This iterative process constructs successive Lie brackets such as $[\mathscr{A} \mathbf{b}^{k-1}, \mathscr{A} \mathbf{x}^{k-1}]$, $[\mathscr{A} \mathbf{b}^{k-1}, \mathscr{A} \mathbf{x}^{k-2} \mathbf{b}]$, and so on. After $n-1$ iterations, $\mathbf{R}$ incorporates all necessary Lie brackets of $\{\mathscr{A} \mathbf{x}^{k-1}, \mathbf{b}\}$ at the origin, based on the established results in nonlinear controllability theory \cite{isidori1985nonlinear}. 
\end{proof}

The proof follows similar steps to those in \cite{chen2021controllability}, but our result applies to general HPDSs without requiring any symmetry assumptions on the dynamic tensors. When $k=2$, \cref{prop: controllability} reduces to the classical Kalman's rank condition for linear dynamical systems. However, computing the controllability matrix $\mathbf{R}$ using the full tensor representation is computationally expensive, as its number of columns grows exponentially with the tensor order $k$, which is estimated as \textcolor{black}{$\mathcal{O}(n^k m^{nk})$.} To address this challenge, we exploit the low-rank structure of TTD- and HTD-based HPDSs to construct the controllability matrix more efficiently.

\begin{corollary}\label{cor: ttdcon}
The TTD-based HPDS \eqref{eq:ttdhpds} with linear input  for even $k$ is strongly controllable if and only if the controllability matrix defined as
\begin{align*}
\textbf R_{TTD}=\begin{bmatrix}\textbf M_0 & \textbf M_1 &\cdots &\textbf M_{n-1}\end{bmatrix},
\end{align*}
where $\textbf M_0=\textbf B$ and each $\textbf M_j$ contains the left singular vectors corresponding to the nonzero singular values of the matrix formed from
\begin{align*}\label{TTDMi}
\left\{\left[\prod_{p=1}^{k-1}\left (\mathscr V^{(p)}\times_2\!\textbf v_p\right )\mathscr V^{(k)}\right]^\top \text{ }\middle\vert\,\!\text{ }\textbf v_p\in\mathrm{col}(\begin{bmatrix}
    \textbf M_0  & \textbf M_1 & \cdots & \textbf M_{j-1} 
\end{bmatrix})\right\},
\end{align*}
for $j=1,2,\dots,n-1$, has full rank $n$.
\end{corollary}
\begin{proof}
Assume that $\textbf X$ is a matrix with a compact SVD given by $\textbf X=\textbf U\bm\Sigma \textbf V^\top$. We first show that the column space of $\{\mathscr A\textbf v_1\textbf v_2\cdots\textbf v_{k-1}\text{ }|\text{ }\textbf v_{p}\in\text{col}(\textbf X)\}$ coincides with that of 
$\{\mathscr A\textbf v_1\textbf v_2\cdots\textbf v_{k-1}\!\text{ }|\text{ }\textbf v_{p}\in\text{col}(\textbf U)\}$. Since $\{\mathscr A\textbf v_1\textbf v_2\!\cdots\!\textbf v_{k-1}\!\text{ }|\text{ }\textbf v_{p}\in\text{col}(\textbf X)\}$ is equivalent to $\textbf A_{(k)}(\textbf X\otimes\textbf X\otimes\stackrel{k-1}{\cdots}\otimes\textbf X)$, and $\{\mathscr A\textbf v_1\textbf v_2\cdots\textbf v_{k-1}\text{ }|\text{ }\textbf v_{p}\in\text{col}(\textbf U)\}$ is equivalent to $\textbf A_{(k)}(\textbf U\otimes\textbf U\otimes\stackrel{k-1}{\cdots}\otimes\textbf U)$, we apply the mixed-product property of Kronecker product
		\begin{align*}
			\textbf A_{(k)}(\textbf X\otimes\textbf X\otimes\overset{k-1}{\cdots}\otimes\textbf X)&=\textbf A_{(k)}(\textbf U\bm\Sigma\textbf V^\top\otimes\textbf U\bm\Sigma\textbf V^\top\otimes\overset{k-1}{\cdots}\otimes\textbf U\bm\Sigma\textbf V^\top)\\
			&=\textbf A_{(k)}(\textbf U\otimes\overset{k-1}{\cdots}\otimes\textbf U)(\bm\Sigma\otimes\overset{k-1}{\cdots}\otimes\bm\Sigma)(\textbf V\otimes\overset{k-1}{\cdots}\otimes\textbf V)^\top.
		\end{align*}
This confirms that $\textbf A_{(k)}(\textbf U\otimes \textbf U \otimes\stackrel{k-1}{\cdots}\otimes\textbf U)$ and $\textbf A_{(k)}(\textbf X\otimes\textbf X\otimes\stackrel{k-1}{\cdots}\otimes\textbf X)$ share the same column space, which further shows that the column space of $\{\mathscr A\textbf v_1\textbf v_2\cdots\textbf v_{k-1}\text{ }|\text{ }\textbf v_{p}\in\text{col}(\textbf X)\}$ coincides with that of 
$\{\mathscr A\textbf v_1\textbf v_2\cdots\textbf v_{k-1}\text{ }|\text{ }\textbf v_{p}\in\text{col}(\textbf U)\}$. 

Without loss of generality, assume $m = 1$ and denote $\mathbf{b} = \mathbf{b}_1$. By using the TTD-based representation of HPDSs  and the properties of tensor vector multiplications, it follows from \cref{prop: controllability} that $[\textbf b, \mathscr A\textbf x^{k-1}]_\textbf{0}=[\textbf b,	[\textbf b, \mathscr A\textbf x^{k-1}]]_\textbf{0}=\textbf{0}$, and
\begin{align*}
[\textbf b,	[\dots,[\textbf b, \mathscr A\textbf x^{k-1}]]]_{\textbf{0}}&=\left(\left.\frac{d}{d\textbf x}\right |_{\textbf x=0}\left[\left(\mathscr V^{(1)}\times_2 \textbf x\right)\prod_{p=2}^{k-2}\left(\mathscr V^{(p)}\times_2 \textbf b\right)\mathscr V^{(k)}\right]^\top\right)\textbf b\\&=\left[\prod_{p=1}^{k-1}\left (\mathscr V^{(p)}\times_2\textbf b\right )\mathscr V^{(k)}\right]^\top.
\end{align*}
The iterative process constructs successive Lie brackets, and after $n-1$ iterations,
 $\mathbf{R}_{\text{TTD}}$ incorporates all necessary Lie brackets of $\{\mathscr{A} \mathbf{x}^{k-1}, \mathbf{b}\}$ at the origin.
\end{proof}

\begin{remark}
If $r=\text{max}\{r_p\text{ }|\text{ } p=1,2,\dots,k-1\}$ is the maximum TT-rank from the TTD of $\mathscr{A}$ and the dimension of the control input is $m$, the time complexity of computing $\textbf{R}_{\text{TTD}}$ in \cref{cor: ttdcon} can be estimated as  $\mathcal O(km^{k-1}nr^2+kn^{k+1}r^2+m^{k-1}n^2+n^{k+2}).$
\end{remark}

\begin{algorithm}[t]
\caption{Computing the TTD-based controllability matrix}\label{alg:cTTD}
\begin{algorithmic}[1]
\Require TT factor tensors  $\mathscr V^{(p)}$ for $p=1,2,\dots,k$ and control matrix $\textbf B$
\Ensure TTD-based controllability matrix $\textbf R_{\text{TTD}}$
\State Initialize $\textbf R_{\text{TTD}}=\textbf B$, $j=0$, and $s=0$
\While{$j<n$  and $s<n$}
\State Generate the set $\Sigma$ containing all selections of $k-1$ elements from $[1,2,\dots,\mathrm{ColumnNum}(\textbf R_{\text{TTD}})]$, allowing repetition 
\For {each permutation $\sigma=\{\sigma_1,\sigma_2,\dots,\sigma_{k-1}\}
\in\Sigma$}
\State Compute $\textbf L\!=\!\left [\prod_{p=1}^{k-1}\left (\mathscr V^{(p)}\!\times_2\textbf R_{\text{TTD}}(:,{\sigma_p})\right )\mathscr V^{(k)}\right]^\top\!$ 
\State Append $\textbf L$ to $\textbf R_{\text{TTD}}$, i.e.,  $\textbf R_{\text{TTD}}=\begin{bmatrix}\textbf R_{\text{TTD}}& \textbf L\end{bmatrix}$
\EndFor
\State Compute compact SVD $\textbf R_{\text{TTD}}=	\textbf U_{\text{T}}\bm\Sigma_{\text{T}} 	\textbf V_{\text{T}}^\top$
and update $s=\mathrm{rank}(	\textbf R_{\text{TTD}})$
\State Set $\textbf R_{\text{TTD}}=	\textbf U_{\text{T}}$ and $j=j+1$
\EndWhile
\end{algorithmic}
\end{algorithm}

If the maximum TT-rank $r$ is small, the computation of the controllability matrix using TTD is significantly lower compared to using the full tensor representation. Additionally, the detailed procedure for constructing the TTD-based controllability matrix $\textbf R_{\text{TTD}}$ is outlined in \cref{alg:cTTD}. Following a similar approach, the controllability matrix based on HTD can be derived accordingly.

\begin{corollary}\label{cor: htdcon}
The HTD-based HPDS \eqref{eq:htdhpds} with linear input for even $k$ is strongly controllable if and only if the controllability matrix defined as
\begin{align*}
\textbf R_{HTD}=\begin{bmatrix}\textbf M_0 & \textbf M_1 &\cdots &\textbf M_{n-1}\end{bmatrix},
\end{align*}
with $\textbf M_0=\textbf B$ and each $\textbf M_j$ contains the left singular vectors corresponding to the nonzero singular values of the matrix formed from
\begin{align*}
    \Big\{(\textbf{U}_k\otimes \textbf{v}_{k-1}^\top\textbf{U}_{k-1}\otimes \cdots\otimes \textbf{v}_{1}^\top\textbf{U}_{1}) (\otimes_{\mathcal{Q}\in\mathcal{G}_{d-1}} \textbf{G}_{\mathcal{G}})\cdots (\otimes_{\mathcal{Q}\in\mathcal{G}_{0}} \textbf{G}_{\mathcal{G}}) \text{ }\\\Big\vert \text{ }
    \textbf v_p\in\mathrm{col}(\begin{bmatrix}
    \textbf M_0  & \textbf M_1 & \cdots & \textbf M_{j-1} 
\end{bmatrix}) \text{ for } p =1,2,\dots,k-1\Big\}
\end{align*}
with $j=1,2,\dots,n-1$, has full rank $n$.
\end{corollary}
\begin{proof}
If $\mathscr A$ is in HTD representation, according to the properties of tensor vector multiplications, it follows that
\begin{align*}
\mathscr A\textbf v_1\textbf v_2\cdots\textbf v_{k-1}&=(\textbf U_k\otimes\textbf U_{k-1}\otimes\cdots\otimes\textbf U_{1})(\otimes_{\mathcal Q\in\mathcal G_{d-1}}\textbf G_\mathcal Q)\cdots(\otimes_{\mathcal Q_\in\mathcal G_{0}}\textbf G_\mathcal Q)\textbf v_1\cdots\textbf v_{k-1}\\
&=(\textbf U_k\otimes \textbf v_{k-1}^\top\textbf U_{k-1}\otimes\cdots\otimes\textbf v_1^\top\textbf U_{1})(\otimes_{\mathcal Q\in\mathcal G_{d-1}}\textbf G_\mathcal Q)\cdots({\otimes}_{\mathcal Q_\in\mathcal G_{0}}\textbf G_\mathcal Q).
\end{align*}
Similar to the proof in \cref{cor: ttdcon}, by iteratively leveraging the HTD-based representation of HPDSs and the properties of tensor vector
multiplications, $\textbf R_{\text{HTD}}$ contains all necessary Lie brackets $\{\mathscr{A} \mathbf{x}^{k-1}, \mathbf{b}\}$ at the origin. 
\end{proof}	

\begin{remark}
If $r=\text{max}\{r_\mathcal Q \text{ }| \text{ }\mathcal Q\in\mathcal T\}$ is the maximum hierarchical  rank from the HTD of $\mathscr{A}$ and the dimension of the control input is $m$, the time complexity of computing $\textbf{R}_{\text{HTD}}$ in \cref{cor: htdcon} is about $\mathcal O(m^{k-1}knr^3+m^{k-1}n^2+n^{k+1}k r^3+n^{k+2}).$ 
\end{remark}

	\begin{algorithm}[t]
		\caption{Computing the HTD-based controllability matrix }\label{alg:cHTD}
		\begin{algorithmic}[1]
			\Require Factor and transfer matrices $\{\textbf U_\mathcal Q, \textbf G_\mathcal Q\}$, a dimension tree $\mathcal T$, and control matrix $\textbf B$
			\Ensure HTD-based controllability matrix $\textbf R_{\mathrm{HTD}}$
			\State Set $\textbf R_{\mathrm{HTD}}=\textbf B$, $j=0$, and $s=0$
			\For{$p=1$ to $k$}
			\State Set $\hat{\textbf U}_p=\textbf U_p$
			\EndFor
			\While{$j<n$  and $s<n$}
			\State Obtain a set $\bm\Sigma$ which contains all selections of $k-1$ elements from  $\{1,2,\dots,\mathrm{ColumnNum}(\textbf R_{\mathrm{HTD}})\}$ with repetition.
			\For {each permutation $\sigma=\{\sigma_1,\dots,\sigma_{k-1}\}\in\Sigma$}
			\For {$p=1$ to $k-1$}
			\State Set $\textbf U_p=\textbf R_{\mathrm{HTD}}(:,v_{\sigma_p})^\top\hat{\textbf U}_p$.
			\EndFor
			\For {$l=\lceil\log_2k\rceil$ to $1$}
			\For {each internal node $\mathcal Q$ on level $l$}
			\State Set $\textbf U_\mathcal Q\!=\!\textbf (\textbf U_{\mathcal Q_r}\!\otimes\!\! \textbf U_{\mathcal Q_l})\textbf G_\mathcal Q$, $\textbf U_{\mathcal Q_l}$ and $\textbf U_{\mathcal Q_r}$ being child nodes of $\textbf U_\mathcal Q$.
			\EndFor
			\EndFor
			\State Augment $\textbf R_{\mathrm{HTD}}=\begin{bmatrix}\textbf R_{\mathrm{HTD}}&\textbf U_{\{1,\dots,k\}}\end{bmatrix}$.
			\EndFor
			\State Compute compact SVD 
			$\textbf R_{\mathrm{HTD}}=\textbf U_{H}\bm\Sigma_{H}\textbf V_{H}^\top$ and update $s=\mathrm{rank}(\textbf R_{\mathrm{HTD}})$
			
            \State Set $\textbf R_{\mathrm{HTD}}=\textbf U_{H}$ and $j=j+1$
			\EndWhile
		\end{algorithmic}
	\end{algorithm} 
    
Similarly, if the maximum hierarchical rank $r$ is small, using HTD to compute the controllability matrix is much efficient compared to using the full tensor representation. Moreover, the detailed steps for computing the HTD-based controllability matrix $\textbf R_{\mathrm{HTD}}$ are provided in \cref{alg:cHTD}. In summary, both TTD- and HTD-based constructions of the controllability matrix offer considerable computational advantages over the full tensor format by exploiting the low-rank structure of the dynamic tensor. The choice between TTD and HTD depends on the underlying structure of the dynamic tensor, where TTD is better suited for chain-like structures, while HTD is more effective for hierarchical structures. Furthermore, both \cref{cor: ttdcon} and \cref{cor: htdcon} can be used to efficiently determine the  accessibility of tensor-based HPDSs of even degrees (i.e., $k$ is odd) according to \cite{chen2021controllability}. 
% With data-driven identification of TTD- and HTD-based HPDSs, we can more efficiently determine the strong controllability of odd-degree HPDSs with linear input from data.

\subsection{Obeservability}\label{sec: observability}
We next introduce tensor-based observability rank  conditions for the following tensor-based HPDS of degree $k-1$ with linear output 
	\begin{align}\label{eq: sysobservability}
		\begin{cases}
			\dot{\textbf x}(t)=\mathscr A\textbf x(t)^{k-1}  \\
			\textbf y(t)=\textbf C\textbf x(t)\end{cases},
	\end{align}
where $\mathscr A\in\mathbb{R}^{n\times n \times\stackrel{k}\cdots\times n}$ is an almost symmetric dynamic tensor, $\textbf x (t)\in  M\subset\mathbb{R}^n$ is the state vector,  $\textbf C\in\mathbb{R}^{l\times n}$ is the output matrix, and $\textbf{y}(t)\in\mathbb{R}^l$ is the output vector. We assume that the system is complete, meaning that for every $\textbf x_0\in M$, there exists a solution $\textbf x(t)\in  M$ satisfying $\textbf x(0)=\textbf x_0$ and $\textbf x(t)\in M$ for all $t\in\mathbb R$. Unfortunately, strong observability is not achievable for HPDSs. Instead, we adopt the concept of local weak observability, as defined in \cite{hermann1977nonlinear}, which provides an algebraic criterion for observability testing using Lie derivatives.

\begin{definition}[\cite{hermann1977nonlinear}]
		Let $U$ be a subset of $M$, and let $\textbf x_0(0), \textbf x_1(0)\in  U$ be two initial states. The states $\textbf x_0(0)$ and $\textbf x_1(0)$ is $U$-distinguishable if their trajectories $\textbf x_0(t)$ and $\textbf x_1(t)$ remain within $ U$ over a finite time interval and satisfy $\textbf C\textbf x_0(t)\neq\textbf C\textbf x_1(t)$ for some $t$ in this time interval.
	\end{definition}
	
	\begin{definition}[\cite{hermann1977nonlinear}]
		The system is locally weakly observable at $\textbf x_0\in  M$ if exists an open neighborhood $U$ of $\textbf x_0$ such that for every open neighborhood $V$ of $\textbf x_0$ contained in $ U$, the set of $V$-distinguishable points from $\textbf x_0$ consists only of $\textbf x_0$ itself.
	\end{definition}

The Lie derivative plays a fundamental role in nonlinear observability analysis because it provides a systematic way to examine how the output function evolves along system trajectories. Given a vector field $\textbf f$ and a scale function $h$, the Lie derivative of $h$ along $\textbf f$ is defined as $L_\textbf fh=(\partial h/\partial\textbf x)\textbf f$, where $\partial h/\partial\textbf x$ denotes the gradient of $h$ with respect to the state  $\textbf{x}$. The Lie derivative satisfies the property
$L_\textbf fDh= DL_\textbf fh,$
where $D$ denotes the differential operator. Higher-order Lie derivatives are defined recursively as $L_\textbf f^jh=L_\textbf f(L_\textbf f^{j-1}h)$ with the base case $L_\textbf f^0h=h$. If $\textbf h$ is a vector field with dimension $l$, its Lie derivative is defined similarly as 
$
    L_\textbf f\textbf h=\begin{bmatrix}
(L_\textbf fh_1)^\top & (L_\textbf fh_2)^\top & \cdots & (L_\textbf fh_l)^\top
\end{bmatrix}^\top.
$
The Lie derivative allows us to construct successive derivatives of the output function with respect to the system dynamics. A key criterion for testing local weak observability of the tensor-based HPDS with linear output \eqref{eq: sysobservability} is verifying whether
$\text{rank}(\textbf O(\textbf x))=n,$ where $\textbf O(\textbf x)$ is the nonlinear observability matrix defined as
\[
\textbf O(\textbf x)=\nabla_\textbf x\left(\begin{bmatrix}
\left(L_{\mathscr A\textbf x^{k-1}}^0\textbf C\textbf x\right)^\top& \left(L_{\mathscr A\textbf x^{k-1}}^1\textbf C\textbf x\right)^\top&\cdots&
\left(L_{\mathscr A\textbf x^{k-1}}^{n-1}\textbf C\textbf x\right)^\top
\end{bmatrix}^\top\right).
\] 
A criterion for computing the observability matrix of tensor-based HPDSs with symmetric dynamic tensors was introduced in \cite{pickard2023observability}. Therefore, we first extend this result to general tensor-based HPDSs with almost symmetric dynamic tensors.

\begin{proposition}\label{prop:observability}
The tensor-based HPDS with linear output \eqref{eq: sysobservability} is locally \\ weakly observable for all $\textbf x$ in an open dense set of $M\subset\mathbb R^n$ if and only if the state-dependent observability matrix defined as
\begin{align*}
		\textbf O(\textbf x)=\left(\begin{matrix}
			\textbf C\\
			\textbf C\textbf A_{(k)}\sum_{q=1}^{k-1}\textbf x^{[q-1]}\otimes \textbf I\otimes\textbf x^{[k-1-q]}\\
			\textbf C\textbf A_{(k)}\textbf F_2\sum_{q=1}^{2k-3}\textbf x^{[q-1]}\otimes \textbf I\otimes\textbf x^{[2k-3-q]}\\\vdots\\
			\textbf C\textbf A_{(k)}\textbf F_2\textbf F_3\cdots\textbf F_{n-1}\sum_{q=1}^{(n-1)k-2n+3}\textbf x^{[q-1]}\otimes \textbf I\otimes\textbf x^{[(n-1)k-2n+3-q]}
		\end{matrix}\right),
	\end{align*}
where
$
			\textbf F_j=\sum_{i=1}^{(j-1)k-2j+3}\textbf I^{[i-1]}\otimes \textbf A_{(k)}\otimes\textbf I^{[(j-1)k-2j+3-i]}
$ (the superscript ``$[\cdot]$'' denotes the Kronecker power),
for $j=2,3,\dots, n-1$, has full rank $n$.
\end{proposition}
\begin{proof}
Since $\mathscr A\textbf x^{k-1}=\textbf A_{(k)}\textbf x^{[k-1]}$, the Lie derivative of the  output $\textbf C\textbf x$ along the vector field $\mathscr A\textbf x^{k-1}$ can be computed as follows:
		\begin{align*} 
			\frac{\partial}{\partial \textbf x}L_{\mathscr A\textbf x^{k-1}}^0\textbf C\textbf x&=\frac{\partial}{\partial \textbf x}\textbf C\textbf x=\textbf C,\\
			\frac{\partial}{\partial \textbf x}L_{\mathscr A\textbf x^{k-1}}^1\textbf C\textbf x&=\frac{\partial}{\partial\textbf x}\left(\frac{d}{dt}\textbf C\textbf x\right)=\frac{\partial}{\partial \textbf x}\textbf C\textbf A_{(k)}\textbf x^{[k-1]}=\textbf C\textbf A_{(k)}\sum_{q=1}^{k-1}\textbf x^{[q-1]}\otimes \textbf I\otimes\textbf x^{[k-1-q]},\\
			\frac{\partial}{\partial \textbf x}L_{\mathscr A\textbf x^{k-1}}^2\textbf C\textbf x&=\frac{\partial}{\partial\textbf  x}\left(\frac{d}{dt}\textbf C\textbf A_{(k)}\textbf x^{[k-1]}\right)=\frac{\partial}{\partial \textbf x}\textbf C\textbf A_{(k)}\textbf F_2\textbf x^{[2k-3]}\\&=\textbf C\textbf A_{(k)}\textbf F_2\sum_{q=1}^{2k-3}\textbf x^{[q-1]}\otimes\textbf I\otimes\textbf x^{[2k-3-q]},\\
			&\quad\vdots\\
			\frac{\partial}{\partial \textbf x}L_{\mathscr A\textbf x^{k-1}}^{n-1}\textbf C\textbf x&=\frac{\partial}{\partial\textbf x}\left(\frac{d}{dt}\textbf C\textbf A_{(k)}\textbf F_2\textbf F_3\cdots\textbf F_{n-2}\textbf x^{[(n-2)k-2n+5]}\right)\\&=\textbf C\textbf A_{(k)}\textbf F_2\textbf F_3\cdots\textbf F_{n-1}\sum_{q=1}^{(n-1)k-2n+3}\textbf x^{[q-1]}\otimes\textbf I\otimes\textbf x^{[(n-1)k-2n+3-q]},
		\end{align*}
where $\textbf F_j$ are defined above. Therefore, the result follows immediately.
\end{proof}	

The tensor-based condition stated in \cref{prop:observability} is simpler to evaluate than computing complex Lie derivatives of vector fields and can be used to assess the local weak observability of general HPDSs. When $k=2$, the result reduces to the classical Kalman's rank condition for observability in linear systems theory.  However, constructing the observability matrix using the full tensor representation is computationally expensive, with a time complexity of approximately \textcolor{black}{$\mathcal{O}(ln^{2nk}+kn^{2nk})$} assuming the dimension of the output is $l$.
 To address this challenge, we propose TTD- and HTD-based methods that leverage factor tensors and matrices to reduce computational costs, analogous to the approach used for the controllability matrix. 

\begin{corollary}\label{cor: ttdo}
The TTD-based HPDS \eqref{eq:ttdhpds} with linear output  is locally \\weakly observable for all $\textbf x$ in an open dense set of $M\subset\mathbb R^n$ if and only if the observability matrix defined as
\[
\textbf O_{\text{TTD}}(\textbf x)=\left(\begin{matrix}
			\textbf C\\
			\textbf C\sum_{q=1}^{k-1} J_1\left(\{\textbf{x},\textbf{x},\stackrel{q-1}\dots,\textbf{x},\textbf I,\textbf x,\textbf{x},\stackrel{k-1-q}{\dots\dots},\textbf{x}\}\right)\\
			\textbf C\sum_{q=1}^{2k-3} J_2\left(\{\textbf{x},\textbf{x},\stackrel{q-1}\dots,\textbf{x},\textbf I,\textbf x,\textbf{x},\stackrel{2k-3-q}{\dots\dots},\textbf{x}\}\right)\\
			\vdots\\
			\textbf	C\sum_{q=1}^{(n-1)k-2n+3} J_{n-1}\left(\{\textbf{x},\textbf{x},\stackrel{q-1}\dots,\textbf{x},\textbf I,\textbf x,\textbf{x},\stackrel{(n-1)k-2n+3-q}{\dots\dots\dots\dots},\textbf{x}\}\right)
		\end{matrix}\right)
\] 
where $J_1 (\mathcal{Z}_{k-1})= \sum_{j_0=1}^{r_0}\dots\sum_{j_k=1}^{r_k}(\mathscr V^{(1)}\times_2 \textbf z_1)_{j_0:j_1}\circ\dots\circ(\mathscr V^{(k-1)}\times_2 \textbf z_{k-1})_{j_{k-2}:j_{k-1}}\circ \mathscr V_{j_{k-1}:j_k}^{(k)}$ and $J_j$ are defined
recursively  for $j=2,3,\dots,n-1$ as
\begin{align*}
    J_j(\mathcal Z_{jk-2j+1}) = 
			\sum_{i=1}^{(j-1)k-2(j-1)+1} J_{j-1}\left(\hat{\mathcal Z}_{(j-1)k-2(j-1)+1}^i\right),
\end{align*}
has full rank $n$. Here, $\mathcal{Z}_p=\{\textbf z_1,\textbf z_2,\dots,\textbf z_p\}$ and $\hat{\mathcal Z}_{(j-1)k-2(j-1)+1}^i$ are derived from $\mathcal Z_{jk-2j+1}$ by merging the $i$th through $(i+k-2)$th elements using the TTD factor tensors such that $\hat{\mathcal Z}_{(j-1)k-2(j-1)+1}^i = \{\textbf{z}_1,\textbf{z}_2,\dots,\textbf{z}_{i-1},\hat{\textbf z},\textbf{z}_{i+k-1},\textbf{z}_{i+k},\dots,\textbf{z}_{jk-2j+1}\}$ where 
$
    \hat{\textbf z}=\sum_{j_0=1}^{r_0}\dots\sum_{j_k=1}^{r_k}(\mathscr V^{(1)}\times_2 \textbf z_{i})_{j_0:j_1}\circ\dots\circ(\mathscr V^{(k-1)}\times_2 \textbf z_{i+k-2})_{j_{k-2}:j_{k-1}}\circ \mathscr V_{j_{k-1}:j_k}^{(k)}.
$
\end{corollary}
\begin{proof}
By utilizing the TTD-based representation of HPDSs and the mixed product property of the Kronecker product, and applying \cref{prop:observability}, we derive that
		\begin{align*}
			& \textbf C\textbf A_{(k)}\sum_{q=1}^{k-1}\textbf x^{[q-1]}\otimes \textbf I\otimes\textbf x^{[k-q-1]}=\textbf C\sum_{q=1}^{k-1}\textbf A_{(k)}\left(\textbf x^{[q-1]}\otimes\textbf I\otimes\textbf x^{[k-q-1]}\right)\\&\ =\textbf C\sum_{q=1}^{k-1} J_1\left(\{\textbf{x},\textbf x,\stackrel{q-1}\dots,\textbf{x},\textbf I,\textbf x, \textbf x,\stackrel{k-1-q}{\dots\dots},\textbf{x}\}\right).
			\end{align*}
		Denote the $i$th element in the set $\{\textbf{x},\textbf{x},\stackrel{q-1}\dots,\textbf{x},\textbf I,\textbf{x}, \textbf x,\stackrel{2k-3-q}{\dots\dots},\textbf{x}\}$ by $\textbf x_{q,i}$. It holds that
        {\small
			\begin{align*}
			& \textbf C\textbf A_{(k)}\textbf F_2\sum_{q=1}^{2k-3}\textbf x^{[q-1]}\otimes \textbf I\otimes\textbf x^{[2k-q-3]}=\textbf C\sum_{q=1}^{2k-3}\textbf A_{(k)}\textbf F_2\left(\textbf x^{[q-1]}\otimes\textbf I\otimes\textbf x^{[2k-q-3]}\right)
			\\&\ =\textbf C\sum_{q=1}^{2k-3}\textbf A_{(k)}\!\sum_{i=1}^{k-1}\!\left(\!\overbrace{\textbf (\textbf x_{q,1}\otimes\textbf x_{q,2}\otimes\stackrel{i-1}\cdots\otimes\textbf x_{q,i-1})\otimes\!\hat{\textbf x}\!\otimes (\textbf x_{q,i+k-1}\otimes\textbf x_{q,i+k}\otimes \stackrel{k-i-1}{\cdots\cdots}\otimes \textbf x_{q,2k-3})}^{k-1\ \text{elements}}\!\right)\\
			&\ =\textbf C\sum_{q=1}^{2k-3}J_2\left(\{\textbf{x},\textbf{x},\stackrel{q-1}\dots,\textbf{x},\textbf I,\textbf x,\textbf{x},\stackrel{2k-3-q}{\dots\dots},\textbf{x}\}\right),
		\end{align*}}
		where $\hat{\textbf x}$ is derived by merging the  $i$th through $(i+k-2)$th elements of $\{\textbf{x},\textbf{x},\stackrel{q-1}\dots,\textbf{x},\textbf{I},\textbf{x},\textbf{x},\stackrel{2k-3-q}{\dots\dots},\textbf{x}\}$ using the TTD factor tensors.
		Therefore, applying the mixed product property recursively leads to
        \begin{align*}
		&\textbf C\textbf A_{(k)}\textbf F_2\textbf F_3\cdots\textbf F_{n-1}\sum_{q=1}^{(n-1)k-2n+3}\textbf x^{[q-1]}\otimes\textbf I\otimes\textbf x^{[(n-1)k-2n-q+3]}\\
        &\quad=\textbf	C\sum_{q=1}^{(n-1)k-2n+3} J_{n-1}\left(\{\textbf{x},\textbf{x},\stackrel{q-1}\dots,\textbf{x},\textbf I,\textbf x,\textbf{x},\stackrel{(n-1)k-2n+3-q}{\dots\dots\dots\dots},\textbf{x}\}\right).
        \end{align*}
        Thus, the result follows immediately.
	\end{proof}

	\begin{remark}	
		If $r=\text{max}\{r_p \text{ }| \text{ } p=1,2,\dots,k-1\}$ is the maximum of the TT-ranks from the TTD of $\mathscr{A}$ and the dimension of the output is $l$, the computational complexity of computing $\textbf O_{\text{TTD}}(\textbf x)$ in \cref{cor: ttdo} can be estimated as \textcolor{black}{$\mathcal O(k^{n}n^{n+1}r^2+ln^2)$}.
	\end{remark}

% Clearly, if the maximum TT-rank $r$ is small, the computation of the observability matrix using TTD is significantly more efficient than using the full tensor representation.
% In \cref{alg: ttdo}, we offer detailed steps for the computation of $J_j(\mathcal Z_{jk-2j+1})$. Using a similar approach, the observability matrix based on the HTD representation can be derived in the same manner.

\begin{corollary}\label{cor: htdo}
The HTD-based HPDS \eqref{eq:htdhpds} with linear output is locally \\weakly observable for all $\textbf x$ in an open dense set of $M\subset\mathbb R^n$ if and only if the observability matrix defined as
\[
\textbf O_{\text{HTD}}(\textbf x)=\left(\begin{matrix}
			\textbf C\\
			\textbf C\sum_{q=1}^{k-1} J_1\left(\{\textbf{x},\textbf{x},\stackrel{q-1}\dots,\textbf{x},\textbf I,\textbf x,\textbf{x},\stackrel{k-1-q}{\dots\dots},\textbf{x}\}\right)\\
			\textbf C\sum_{q=1}^{2k-3} J_2\left(\{\textbf{x},\textbf{x},\stackrel{q-1}\dots,\textbf{x},\textbf I,\textbf x,\textbf{x},\stackrel{2k-3-q}{\dots\dots},\textbf{x}\}\right)\\
			\vdots\\
			\textbf	C\sum_{q=1}^{(n-1)k-2n+3} J_{n-1}\left(\{\textbf{x},\textbf{x},\stackrel{q-1}\dots,\textbf{x},\textbf I,\textbf x,\textbf{x},\stackrel{nk-2n+1-q}{\dots\dots\dots},\textbf{x}\}\right)
		\end{matrix}\right)
\] 
where $J_1(\mathcal{Z}_{k-1})= \phi\left((\textbf U_k\otimes \textbf z_{k-1}^\top\textbf U_{k-1}\otimes\cdots\otimes\textbf z_1^\top\textbf U_{1})(\otimes_{\mathcal Q\in\mathcal G_{d-1}}\textbf G_\mathcal Q)\cdots(\otimes_{\mathcal Q_\in\mathcal G_{0}}\textbf G_\mathcal Q)\right)$, and $J_j$ are defined
recursively  for $j=2,3,\dots,n-1$ as
\begin{align*}
    J_j(\mathcal Z_{jk-2j+1}) = 
			\sum_{i=1}^{(j-1)k-2(j-1)+1} J_{j-1}\left(\hat{\mathcal Z}_{(j-1)k-2(j-1)+1}^i\right),
\end{align*}
has full rank $n$. Here, $\phi$ reshapes a vector of size $N$ into a matrix of size $n\times (N/n)$ (assuming $N$ is divisible by $n$), $\mathcal{Z}_p=\{\textbf z_1,\textbf z_2,\dots,\textbf z_p\}$, and $\hat{\mathcal Z}_{(j-1)k-2(j-1)+1}^i$ are derived from $\mathcal Z_{jk-2j+1}$ by merging the $i$th through $(i+k-2)$th elements using the HTD factor matrices such that $\hat{\mathcal Z}_{(j-1)k-2(j-1)+1}^i = \{\textbf{z}_1,\textbf{z}_2,\dots,\textbf{z}_{i-1},\hat{\textbf z},\textbf{z}_{i+k-1},\textbf{z}_{i+k},\dots,\textbf{z}_{jk-2j+1}\}$ where 
$
    \hat{\textbf z}=\phi\left((\textbf U_k\otimes \textbf z_{i+k-2}^\top\textbf U_{k-1}\otimes\cdots\otimes\textbf z_i^\top\textbf U_{1})(\otimes_{\mathcal Q\in\mathcal G_{d-1}}\textbf G_\mathcal Q)\cdots(\otimes_{\mathcal Q_\in\mathcal G_{0}}\textbf G_\mathcal Q)\right).
$
\end{corollary}
 \begin{proof}
Assume that the dynamic tensor $\mathscr A$ is in HTD form. Due to the almost symmetric nature of $\mathscr A$, it holds that $\textbf U_1=\textbf U_2=\dots=\textbf U_{k-1}$. By applying  tensor vector/matrix multiplications, we obtain
		\begin{align*}
			&\text{vec}\left(\textbf A_{(k)}\left(\textbf x^{[q-1]}\otimes\textbf I\otimes\textbf x^{[k-q-1]}\right)\right)\\&\ =(\textbf U_{k}\otimes_{s=1}^{q-1}\textbf x^\top\textbf U_{1}\otimes \textbf U_1\otimes_{s=1}^{k-q-1}\textbf x^\top\textbf U_{1})(\otimes_{\mathcal Q\in\mathcal G_{d-1}}\textbf G_\mathcal Q)\cdots(\otimes_{\mathcal Q\in\mathcal G_{0}}\textbf G_\mathcal Q)\in\mathbb{R}^{n^2\times 1}.
		\end{align*}
		It follows that
        \begin{align*}
			&\textbf A_{(k)}\left(\textbf x^{[q-1]}\otimes\textbf I\otimes\textbf x^{[k-q-1]}\right)=J_1\left(\{\textbf{x},\textbf{x},\stackrel{q-1}\dots,\textbf{x},\textbf I,\textbf x,\textbf{x},\stackrel{k-1-q}{\dots\dots},\textbf{x}\}\right)\in\mathbb{R}^{n\times n}.
		\end{align*}
		Similar to the proof in \cref{cor: ttdo}, by recursively applying the mixed product property and leveraging the fact that $\{\textbf{x},\textbf{x},\stackrel{q-1}\dots,\textbf{x},\textbf I,\textbf x,\textbf x,\stackrel{jk-2j+1-q}{\dots\dots\dots},\textbf{x}\}$ for $j=1,2,\dots,n$ contain a single matrix (i.e., the identity), it follows that 
			\begin{align*}
			&\textbf C\textbf A_{(k)}\textbf F_2\textbf F_3\cdots\textbf F_j\sum_{q=1}^{jk-2j+1}\textbf x^{[q-1]}\otimes\textbf I\otimes\textbf x^{[jk-2j-q+1]}\\
            &\quad=\textbf	C\sum_{q=1}^{jk-2j+1}J_j\left(\{\textbf{x},\textbf{x},\stackrel{q-1}\dots,\textbf{x},\textbf I,\textbf x,\textbf{x},\stackrel{jk-2j+1-q}{\dots\dots\dots},\textbf{x}\}\right).
		\end{align*}
        Therefore, the result follows immediately. 
	\end{proof}	

	\begin{remark}
		If $r=\text{max}\{r_\mathcal Q \text{ }| \text{ }\mathcal Q\in\mathcal T\}$ is the maximum hierarchical  rank from the HTD of $\mathscr{A}$ and the dimension of the output is $l$, the computational complexity of computing $\textbf O_{\text{HTD}}(\textbf x)$ in \cref{cor: htdo} can be estimated as \textcolor{black}{$\mathcal O(k^{n}n^{n+1}r^3+ln^2)$}.
	\end{remark}

Both results are similar to those for controllability. When the maximum TT-rank or hierarchical rank $r$ is small, calculating the observability matrix using TTD or HTD is significantly more efficient than using the full tensor representation. Moreover, the detailed steps for constructing  $J_j(\mathcal Z_{jk-2j+1})$ using TTD and HTD are provided in \cref{alg: ttdo} and \cref{alg: htdo}, respectively.

\section{Numerical examples}\label{sec: simulation}
We proceed to illustrate our framework through the following numerical experiments. All examples were conducted on a machine equipped with an Apple M3 chip, 16GB memory and MATLAB R2023b, utilizing Tensor Toolbox 3.6 \cite{kolda2023matlab}, TT Toolbox \cite{osetoolbox}, and HTD Toolbox \cite{kressner2012htucker}.  

\subsection{Memory consumption comparison}\label{sec: memory}
In this experiment, we compared the total number of model parameters required for the full, TTD-based, and HTD-based representations of tensor-based HPDSs. The dynamic tensors $\mathscr{A}\in\mathbb{R}^{n\times n\times \stackrel{k}{\cdots}\times n}$ were randomly generated using three distinct generation schemes:
\begin{itemize}
\item Symmetric tensor generation: A random symmetric tensor was generated and subsequently decomposed into TTD and HTD representations.
\item Low TT-Rank tensor generation: A random tensor with inherently low TT-ranks was directly generated, from which the full tensor and its HTD representation were obtained. 
\item Low hierarchical-rank tensor generation:
A random tensor with inherently low hierarchical ranks was directly generated, from which the full tensor and its TTD representation were obtained.
\end{itemize}  
 To evaluate these representations, we set the dimension $n=2$ and consider orders $k=5$, $10$, and $15$, respectively.
The memory usage of these system representations was computed based on their respective decompositions. The comparison of the number of parameters required for each representation is illustrated in \cref{fig: memroy_comparation}. The results demonstrate that both TTD and HTD representations effectively reduce the number of model parameters across all tensor generation schemes. Notably, the reduction is even more significant when the dynamic tensors inherently have low TT-ranks or hierarchical ranks. The choice between TTD and HTD representations depends on the specific values of the TT-ranks and hierarchical ranks, as these directly influence the trade-off between parameter reduction and approximation accuracy.

\begin{figure}[t]
\begin{center}			
\includegraphics[width=\textwidth]{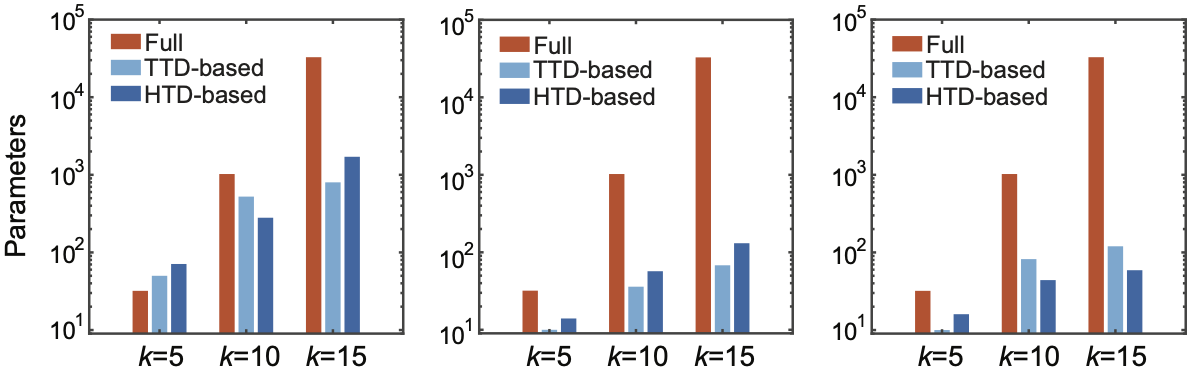}    % The printed column width is 8.4 cm.
\vspace{-15pt}
\caption{Comparison of memory consumption for the full, TTD-based, and HTD-based representations of tensor-based HPDSs across different tensor orders. The three panels, from left to right, represent symmetric tensor generation, low TT-rank tensor generation, and low hierarchical-rank tensor generation, respectively.}
\end{center}
\label{fig: memroy_comparation}
\vspace{-10pt}
\end{figure}

\subsection{Computation time comparison}

In this example, we constructed the controllability matrices for tensor-based HPDSs using the full, TTD-based, and HTD-based representations and evaluated their computational efficiency. The controllability matrix for each representation was constructed according to \cref{prop: controllability}, \cref{cor: ttdcon}, and \cref{cor: htdcon}. The dynamic tensors $\mathscr{A}\in\mathbb{R}^{n\times n\times \stackrel{k}{\cdots}\times n}$ were randomly generated using the same three schemes: symmetric tensor generation, low TT-rank tensor generation, and low hierarchical-rank tensor generation. The control matrices $\textbf B\in\mathbb{R}^{n\times 5}$ were also randomly generated. We considered three scenarios where the state dimension $n$ was set to $n = 5$, $10$, and $15$, and for each case, the value of order $k$ was incrementally increased. This setup allows us to systematically evaluate how computation time scales with both the state dimension and order across different generation schemes.

The computational time required for each representation was recorded and compared, as shown in \cref{fig: computation_time}. The full representation serves as a baseline for evaluating the effectiveness of the TTD-based and HTD-based representations in reducing computational complexity. Our results demonstrate that both the TTD and HTD representations significantly reduce computation time compared to the full representation, owing to their ability to efficiently exploit the low-rank structure of the dynamic tensors. The relative performance between the TTD- and HTD-based representations varies depending on the values of the TT-ranks and hierarchical ranks. In particular, when the tensor order $k$ is small, the TTD-based representation generally outperforms the HTD-based representation. These findings highlight the importance of selecting an appropriate tensor decomposition strategy based on the specific structural properties of the dynamic tensor.

\begin{figure}[t]
\begin{center}		
\includegraphics[width=\textwidth]{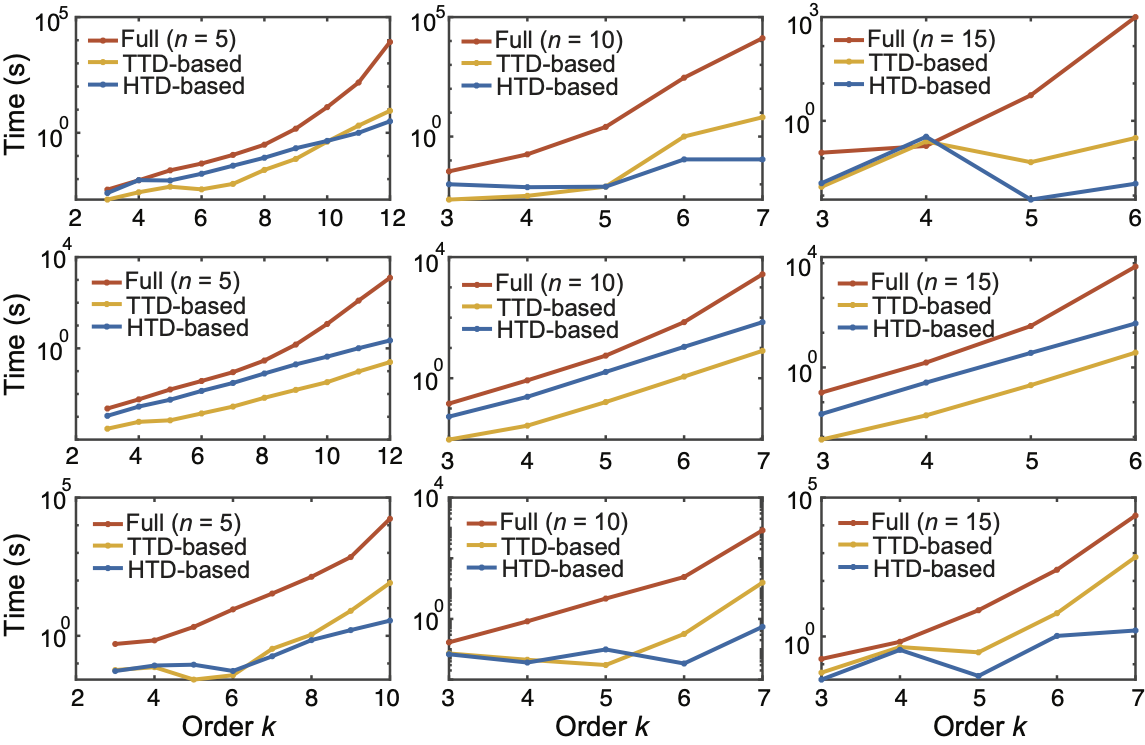}  
\vspace{-20pt}
\caption{Computational time comparison for full, TTD-based, and HTD-based representations of tensor-based HPDSs across varying tensor orders and dimensions. The panels in the first row display tensors generated with symmetry, the second row shows tensors generated with low TT-ranks, and the third row presents tensors generated with low hierarchical ranks.}
\label{fig: computation_time}
\end{center}
\vspace{-10pt}
\end{figure}

\section{Conclusion}\label{sec:conclusion}
In this article, we tackled the memory and computational challenges inherent in analyzing tensor-based HPDSs by introducing a novel computational framework built on TTD and HTD. We developed data-driven system identification techniques using TTD and HTD from time-series data. By leveraging the low-rank structures of the TTD-based and HTD-based representations, we derived efficient necessary and sufficient conditions for the controllability and observability of tensor-based HPDSs, significantly improving computational efficiency and reducing memory demands. Numerical simulations confirmed the effectiveness of our proposed framework. Future work will focus on extending these techniques to stochastic and hybrid systems, integrating data-driven tensor-based methods with machine learning algorithms to enhance scalability, adaptability, and real-time decision-making. Additionally, we plan to investigate the robustness and efficiency of these approaches in high-dimensional and uncertain environments, enabling their application to complex, higher-order networked systems. This work establishes a foundation for further advancements in tensor-based dynamical system analysis, bridging the gap between theory and real-world applications across fields such as biology, engineering, neuroscience, and other interdisciplinary domains.

% \section*{Acknowledgment} The authors used ChatGPT-4o to polish the text for spelling, grammar, and general style.

\bibliographystyle{plain}
\bibliography{reference}

\appendix
\section{Algorithms}
    \begin{algorithm}[h]
		\caption{TTD-based $\text{RecursiveJ}(j,k,\mathcal Z_{jk-2j+1})$}\label{alg: ttdo}
		\begin{algorithmic}[1]
			\If {$j=1$}
			\State Set \[J_j(\textbf x)=\sum_{j_0=1}^{r_0}\dots\sum_{j_k=1}^{r_k}(\mathscr V^{(1)}\times_2\textbf z_1)_{j_0:j_1}\circ\dots\circ(\mathscr V^{(k-1)}\times_2\textbf z_{k-1})_{j_{k-2}:j_{k-1}}\circ \mathscr V_{j_{k-1}:j_k}^{(k)}\]
            
			\Return $ J_j(\textbf x)$
			\EndIf
			\State Set $ J_j(\textbf x)=0$
			\For {$i=1$ to $(j-1)k-2(j-1)+1$}
			\State Compute \[\hat{\textbf z}=\sum_{j_0=1}^{r_0}\dots\sum_{j_k=1}^{r_k}(\mathscr V^{(1)}\times_2\textbf z_i)_{j_0:j_1}\circ\dots\circ(\mathscr V^{(k-1)}\times_2\textbf z_{i+k-2})_{j_{k-2}:j_{k-1}}\circ \mathscr V_{j_{k-1}:j_k}^{(k)}\]
		\State Set $\hat{\mathcal Z}_{(j-1)k-2(j-1)+1}^i = \underbrace{\{\textbf{z}_1,\dots,\textbf{z}_{i-1},\hat{\textbf z},\textbf{z}_{i+k-1},\dots,\textbf{z}_{jk-2j+1}\}}_{(j-1)k-2(j-1)+1 \text{ elements}}$
			\State Set $ J_j(\textbf x )= J_j(\textbf x)+$ $\text{RecursiveJ}\left({j-1},k,\hat{\mathcal Z}_{(j-1)k-2(j-1)+1}^i\right)$
			\EndFor
			\Return $ J_j(\textbf x)$
		\end{algorithmic}
	\end{algorithm}

	\begin{algorithm}[ht]

		\caption{HTD-based $\text{RecursiveJ}(j,k,\mathcal Z_{jk-2j+1})$\label{alg: htdo}}
		\begin{algorithmic}[1]
          \For{$p=1$ to $k$}
			\State Set $\hat{\textbf U}_p=\textbf U_p$
			\EndFor
			\If {$j=1$}
			\For {$p=1$ to $k-1$}
			\State Set $\textbf U_p=\textbf z_p^\top\hat{\textbf U}_p$
			\EndFor
			\For {$l=\lceil\log_2k\rceil$ to $1$}
			\For {each internal node $\mathcal Q$ on level $l$}
			\State Set $\textbf U_\mathcal Q=(\textbf U_{\mathcal Q_r}\otimes \textbf U_{\mathcal Q_l})\textbf G_\mathcal Q$, $\textbf U_{\mathcal Q_l}$ and $\textbf U_{\mathcal Q_r}$ being child nodes of $\textbf U_\mathcal Q$.
			\EndFor
			\EndFor
			\State Set $ J_j(\textbf x)=\text{reshape}\left(\textbf U_{\{1,2,\dots,k\}},\left[n,\text{length}\left(\textbf U_{\{1,2,\dots,k\}}\right)/n\right]\right)$
			\EndIf\\
			\Return $  J_j(\textbf x)$
			\State Set $ J_j(\textbf x)=0$
			\For {$i=1$ to $(j-1)k-2(j-1)+1$}
               \For {$p=1$ to $k-1$}
			\State Set $\textbf U_p=\textbf z_{i+p-1}^\top\hat{\textbf U}_p$.
			\EndFor
			\For {$l=\lceil\log_2k\rceil$ to $1$}
			\For {each internal node $\mathcal Q$ on level $l$}
			\State Set $\textbf U_\mathcal Q=(\textbf U_{\mathcal Q_r}\otimes \textbf U_{\mathcal Q_l})\textbf G_\mathcal Q$,  $\textbf U_{\mathcal Ql}$ and $\textbf U_{\mathcal Qr}$ being child nodes of $U_\mathcal Q$
			\EndFor
			\EndFor
                \State Set $\hat{\textbf z}=\text{reshape}\left(\textbf U_{\{1,2,\dots,k\}},\left[n,\text{length}\left(\textbf U_{\{1,\dots,k\}}\right)/n\right]\right)$
			\State Set $\hat{\mathcal Z}_{(j-1)k-2(j-1)+1}^i = \underbrace{\{\textbf{z}_1,\dots,\textbf{z}_{i-1},\hat{\textbf z},\textbf{z}_{i+k-1},\dots,\textbf{z}_{jk-2j+1}\}}_{(j-1)k-2(j-1)+1 \text{ elements}}$
			\State Set $ J_j(\textbf x)= J_j(\textbf x)+\text{RecursiveJ}\left(j-1,k,\hat{\mathcal Z}_{(j-1)k-2(j-1)+1}^i\right)$.
			\EndFor\\
			\Return $ J_j(\textbf x)$
		\end{algorithmic}
	\end{algorithm}

\end{document}